\documentclass{amsart}
\usepackage[T1]{fontenc}

\usepackage{amsthm}
\usepackage{tikz}
\usepackage{amsmath}
\usepackage{amsfonts}
\usepackage{leftindex}
\usepackage{physics}
\usepackage{microtype}
\usepackage{lmodern}
\usepackage{tikz-cd}
\usepackage{mathtools}
\usepackage{amssymb}
\usepackage{mathrsfs}
\usepackage{hyperref}
\usepackage{cleveref}
\hypersetup{
    colorlinks=true,
    linkcolor=black,
    filecolor=magenta,      
    urlcolor=red,
}

\newtheorem{theorem}{Theorem}

\newtheorem{lemma}[theorem]{Lemma}
\newtheorem{definition}[theorem]{Definition}
\newtheorem{proposition}[theorem]{Proposition}
\newtheorem{corollary}[theorem]{Corollary}
\newtheorem{remark}[theorem]{Remark}

\newcommand{\RR}{\mathbb{R}}
\newcommand{\ZZ}{\mathbb{Z}}

\newcommand{\CC}{\mathbb{C}}

\newcommand{\mfc}{\mathfrak{c}}
\newcommand{\mfg}{\mathfrak{g}}

\DeclareMathOperator{\Hom}{Hom}

\DeclareMathOperator{\End}{End}

\DeclareMathOperator{\Fl}{Fl}

\title{Kazhdan-Lusztig map and Langlands duality}

\author{Anlong Chua}

\begin{document}
\maketitle

\begin{abstract}
Let $G$ be a connected reductive group over $\CC$ with Weyl group $W$. Following a suggestion of Bezrukavnikov, we define a map from two-sided cells to conjugacy classes in $W$ using the geometry of the affine flag variety. This is an affine analog of the classical story of two-sided cells of $W$, special nilpotent orbits and special representations of $W$. The proof involves studying Springer representations appearing in the cohomology of affine Springer fibers. This relies on tools developed by Yun in his papers on Global Springer theory. Our result provides evidence for a conjecture of Lusztig on strata in a connected reductive group.
\end{abstract}

\section{Introduction}

\subsection{The classical setting}\label{subsect:intro_classical}

Let $G$ be a connected, reductive group over $\CC$ with Lie algebra $\mfg$. Fix a Borel subgroup $B\subset G$ and maximal torus $T \subset G$. Denote their Lie algebras by $\mathfrak{b}$, $\mathfrak{t}$ respectively. Let $W$ be the Weyl group of $G$, and denote the set of conjugacy classes in $W$ by $\underline{W}$. Let $\mathcal{N} \subset \mfg$ denote the nilpotent cone, and $\underline{\mathcal N}$ be the finite set of nilpotent orbits.

Let $\mathbb{X}_*(T) = \operatorname{Hom}(\mathbb{G}_m,T)$, $\mathbb{X}^*(T) = \operatorname{Hom}(T,\mathbb{G}_m)$ be the character and cocharacter lattices of $T$ respectively. Let $\Phi^+ \subset \Phi \subset \mathbb{X}^*(T)$ and $\Phi^{\vee,+} \subset \Phi^\vee \subset \mathbb{X}_*(T)$ be the based root and coroot systems.

In the introduction, we assume for simplicity that $G$ is simply connected.

Denote $\operatorname{rank} G = r$, and fix a choice of simple roots $\Sigma = \{\alpha_1, \dots, \alpha_r\} \subset \Phi^+$. This gives $(W, \Sigma)$ the structue of a Coxeter group, generated by the simple reflections corresponding to the simple roots.

In \cite{kazhdan1979representations}, Lusztig defines a partition of $W$ (as a set) into two-sided cells. We denote the set of two-sided cells of $W$ by $\operatorname{Cells}(W)$.

The flag variety $F = G/B$ has a stratification by affine spaces, indexed by elements of $W$:
\[
F = \bigsqcup_{w \in W} BwB/B.
\]
We denote the IC sheaf of the stratum $BwB/B$ by $IC_w$. If $j \colon BwB/B \to F$ is its locally closed embedding, then $IC_w = j_{!*}\CC[\ell(w)]$. This is a $B$-equivariant perverse sheaf on $F$. Its singular support, which we denote $SS(IC_w)$, is a certain conical subset of $T^*F = G \times_B \mathfrak{n}$. It hence comes equipped with a projection to $\mathcal{N}$. Denote the image of this projection by $\operatorname{AV}(w)$. In \cite{barbasch_primitive_1983}, Barbasch and Vogan show that $\operatorname{AV}(w) = \overline{O} \cap \mathfrak{n}$. Therefore, there is a well-defined largest orbit in the image of $\operatorname{AV}(w)$. 

It is known that $\operatorname{AV}$ is constant on two-sided cells of $W$. It follows that, $\operatorname{AV}$ gives a well defined map $\operatorname{Cells}(W) \to \underline{\mathcal{N}}$. 

Through the work of Barbasch-Vogan, Lusztig, Joseph, Schmid-Vilonen \cite{lusztig_characters_1984,joseph_associated_1985,LUSZTIG20093418,schmid2000characteristic}, it is known that there is a commutative diagram of bijections
\begin{center}
\begin{tikzcd}\label{diagram:finite}
\operatorname{Cells}(W)\arrow[rr, "\operatorname{AV}", ] \arrow[dr, leftrightarrow]& & \underline{\mathcal{N}}^{sp}\\
& \operatorname{Rep}^{sp}(W).\arrow[ur, leftrightarrow , "\operatorname{Spr}",sloped] &
\end{tikzcd}
\end{center}
Here, $\underline{\mathcal{N}}^{sp} \subset \underline{\mathcal{N}}$ is the set of \textit{special} nilpotent orbits. Similarly, $\operatorname{Rep}^{sp}(W) \subset \operatorname{Rep}(W)$ is the set of \textit{special} $W$ representations. These notions were defined by Lusztig \cite{lusztig1979class}. We review them in \Cref{sect:reformulation}.

The arrow labelled $\operatorname{Spr}$ denotes Lusztig's generalized Springer correspondence \cite{lusztig1981green}. The remaining bijection above is Kazhdan-Lusztig's construction of cell representations \cite{lusztig_cells_1989}. These are the same as Goldie rank representations defined by Joseph \cite{joseph1980goldie1,joseph1980goldie2}, \cite[Chapter 5]{lusztig_characters_1984}.

\subsection{An affine analog} \label{sect:intro_affine}

Let $F = \CC((t))$ be the field of Laurent series in $t$. The subring of Taylor series $\mathcal{O} = \CC[[t]]$ is its ring of integers with respect to the $t$-adic valuation. Let $X$ be an affine scheme of finite type over $\CC$. We denote by $LX$ the functor $LX(R) \mapsto X(R((t)))$ on the category of $\CC$-algebras. This is representable by an Ind-scheme. Similarly, define $L^+X$ to be the functor $L^+X(R) \mapsto X(R[[t]])$. It is representable by a scheme (not of finite type). We have the reduction morphism $\operatorname{ev}_t \colon L^+X \to X$ given by reduction mod $t$.

In particular, we have $ \operatorname{ev}_t^{-1}(B)\eqqcolon I \subset L^+ G \subset LG$. Let $I \subset P \subset LG(\CC)$ be a standard parahoric subgroup. We define the corresponding affine partial flag variety $\Fl^P$ to be the sheafification of the functor $\Fl^P(R) \mapsto LG(R) / P(R)$ on the category of $\CC$-algebras. When $P = I$, we abbreviate $\Fl^I = \Fl$, and call this the affine flag variety of $G$. 

Let $\ZZ\Phi^\vee \subset \mathbb{X}_*(T)$ be the coroot lattice. Define
\[
W_{\operatorname{aff}} \coloneqq \mathbb{Z}\Phi^\vee \rtimes W , \qquad \widetilde{W} = \mathbb{X}_*(T) \rtimes W  
\]
to be the affine Weyl group and extended affine Weyl group of $G$. If $G$ is almost simple, $W_{\operatorname{aff}}$ is a Coxeter group with simple reflections $\Sigma_{\operatorname{aff}} \coloneqq \Sigma \cup \{s_0\}$ where $s_0$ is the reflection corresponding to the affine root. In this case, standard parahoric subgroups are in bijection with proper subsets of $\Sigma_{\operatorname{aff}}$.

Like in the classical case, there is an Iwahori decomposition
\[
\Fl = \bigsqcup_{w \in \widetilde{W}} IwI/I
\]
with $IwI/I \cong \mathbb{C}^{\ell(w)}$.

Let $\langle - , -\rangle$ denote the Killing form on $\mathfrak{g}$. Using the invariant form $\operatorname{res}\langle - , -\rangle \frac{dt}{t}$ we identify the fiber over $gI$ with $\operatorname{Ad}(g)(\operatorname{Lie} I^+)$, where $I^+$ is the prounipotent radical of $I$. Let $\pi\colon T^*\Fl\to L\mfg$ be the projection.

We attempt to adapt Barbasch-Vogan's associated variety results from the classical case. Let $IC_w = j_{!*} \CC[\ell(w)]$. Denote its singular support by $SS(IC_w) \subset T^*\Fl$. It inherits the projection $\pi\colon T^*\Fl \to  L\mfg$. We will show:
\begin{theorem}\label{thm:intro_irred}
    \begin{enumerate}
        \item The set $\pi( SS(IC_w)) \subset L\mfg$ contains a dense open subset of regular semisimple, topologically nilpotent elements which can be $LG$-conjugated into the same Cartan subalgebra of $L\mfg$.
        \item The $LG$-conjugacy type of this Cartan subalgebra is independent of the choice of dense open subset.
    \end{enumerate}
\end{theorem}

Recall \cite[Lemma 1.2]{kazhdan1988fixed} that $LG$-conjugacy classes of Cartan subalgebras of $L\mfg$ are parametrized by conjugacy classes in $W$. If $\gamma \in L\mfg$ is regular semisimple, we will say that $\gamma$ is \textit{of type $[w]$} if it lies in a Cartan subalgebra corresponding to the conjugacy class $[w]$.

Denote the set of conjugacy classes in $W$ by $\underline{W}$. It follows from \Cref{thm:intro_irred} that we have a well defined map $\widetilde{\operatorname{AV}}\colon \widetilde{W} \to \underline{W}$. The definition of $\widetilde{\operatorname{AV}}$ was suggested to us by Roman Bezrukavnikov. We now state our main theorem:
\begin{theorem}\label{thm:intro_diagram}
    The map $\widetilde{AV}$ is constant on two-sided cells of $\widetilde{W} = W_{\operatorname{aff}}$. It fits into a commutative diagram
    \begin{center}
    \begin{tikzcd}
    \operatorname{Cells}(W_{\operatorname{aff}})\arrow[rr, "\underline{c}" ,leftrightarrow] \arrow[dr,"\widetilde{\operatorname{AV}}"]& & \underline{\mathcal{N}_{G^\vee}}\\
    & \underline{W}.\arrow[ur, hookleftarrow , "\operatorname{KL}_{G^\vee}",sloped] &
    \end{tikzcd}
    \end{center}
\end{theorem}
Let us explain the maps in the above diagram. The arrow labelled $\underline{c}$ refers to Lusztig's bijection of two-sided cells in $W_{\operatorname{aff}} = \widetilde{W}$ with nilpotent orbits in $G^\vee$, the Langlands dual of $G$. The arrow $\operatorname{KL}_{G^\vee}$ refers to the map constructed by Kazhdan and Lusztig in \cite[Section 9]{kazhdan1988fixed}. We recall the definitions of these maps in \Cref{sect:main_thm}.

\subsection{Acknowledgements} I would like to sincerely thank my advisor Roman Bezrukavnikov for suggesting the problem and his excellent guidance throughout the project. I would like to express my deep gratitude to Zhiwei Yun for explaining various technical details related to his papers on global Springer theory and minimal reduction type. I would also like to thank him for his help and suggestions with various arguments in the paper. Additionally, I would like to thank George Lusztig for explaining to me some aspects of his work on cells in Weyl groups, and for pointing out a typographical error in an earlier version of the paper. I would also like to thank Minh-T\^{a}m Trinh for pointing me towards Yun's papers on global Springer theory. Finally, I am grateful to Ivan Losev, Calder Morton-Ferguson, Pablo Boixeda Alvarez, Vasily Krylov and Ivan Karpov for helpful discussions.

\section{Arc spaces}

Recall the notations $L(-)$, $L^+(-)$ from \Cref{sect:intro_affine}.

\subsection{The Chevalley base} Let $\mathfrak{c} = \mfg//G = \mathfrak{t}//W$. Let $\chi = \mfg \to \mfc$ be the quotient. It induces a map $L\mfg \to L\mathfrak{c}$, which we will also denote by $\chi$.

Let $L^\heartsuit \mfg \subset L^+ \mfg$ be the subscheme of topologically nilpotent, regular semisimple elements. We will use $L^\heartsuit( - ) $ to denote topologically nilpotent, regular semisimple elements in various arc spaces related to $\mfg$. 

For any $[w] \in \underline{W}$, let $(L^\heartsuit \mfg)_{[w]} \subset L\mfg$ be the subset of topologically nilpotent, regular semisimple elements of type $[w]$. Let $(L^\heartsuit\mfc)_{[w]} = \chi(L^\heartsuit\mfg)_{[w]}$ its image in $L\mfc$. Yun shows in \cite[Section 4]{yun2021minimal} that $(L^\heartsuit\mfc)_{[w]}$ is a scheme with reasonable properties. For example, it is locally closed (see \Cref{sect:arc_topology}) in $L^+\mfc$.

Let $\Delta \colon \mathfrak{t} \to \mathbb{A}^1$,
\begin{equation}\label{eqn:discriminant}
    \Delta(x) = \prod_{\alpha \in \Phi} \alpha(x)
\end{equation}
be the discriminant function. We extend this to a function on $\mfg$ by $\operatorname{Ad}(G)$-invariance, and then to a function on $L\mfg$ by extension of scalars.

Note that $\gamma \in L\mfg$ is conjugate to an element in $\mathfrak{t}\left(\CC((t^{1/m}))\right)$ after passing to the algebraic closure $\overline{\CC((t))} = \bigcup_{k \geq 1} \CC((t^{1/k}))$. Extending the $t$-adic valuation so that $t^{m/n} = \frac{m}{n}$, we can make sense of $\operatorname{val}(\Delta(\gamma))$. 

Define
\begin{equation}\label{eqn:springer_dimension}
\delta(\gamma) = \frac{\operatorname{val}(\Delta(\gamma)) - (\dim \mathfrak{t} - \dim \mathfrak{t}^w)}{2}
\end{equation}
where $\mathfrak{t}^w$ is the subspace where $w$ acts by eigenvalue 1. 

We now recall a definition due to Yun \cite[Definition 1.6]{yun2021minimal}:
\begin{definition}
    For $[w] \in W$, let $\delta_{[w]} = \min \{ \delta(\gamma) \mid \gamma \in (L^\heartsuit \mfg)_{[w]}\}$. An element $\gamma \in L^\heartsuit \mfg$ is called shallow of type $[w] \in \underline{W}$ if $\gamma$ is of type $[w]$, and $\delta(\gamma) = \delta_{[w]}$.
\end{definition}
We denote the set of shallow elements of type $[w]$ by $(L^\heartsuit \mfg)^{sh}_{[w]}$. Let $(L^\heartsuit \mfc)^{sh}_{[w]}$ be its image under $\chi$. Similarly, this is a locally closed subscheme of $L^+ \mfc$. 

Similarly, let $P \subset LG(\CC)$ be a parahoric subgroup, and $O \subset \operatorname{Lie}G_P$ be a nilpotent orbit. We denote by $(L^\heartsuit \mfc)_O^P $ the image of $ (O+P^+) \cap L^\heartsuit \mfg$ under $\chi$. Note that by \Cref{lem:parahoric_nilpotent}, every element in $O+P^+$ is already topologically nilpotent. Similarly, we define $(L^\heartsuit\mfc)_{\overline{O}}^P$ as the image of $(\overline{O} + P^+) \cap L^\heartsuit \mfg$. 

\subsection{Topological notions for arc spaces}\label{sect:arc_topology} Let $X$ be an affine scheme of finite type over $\CC$. Define, for any $n\geq 0$, $L_n^+ X$ to be the truncated arc space representing $R \mapsto X(R[[t]]/t^{n+1})$. Then $L^+X = \varprojlim_n L_n^+ X$. Let $\pi_n \colon L^+ X \to L_n^+ X $ be the projection. 

We call $Z \subset L^+X(\CC)$ ``fp-constructible'' (respectively, ``fp-open'', ``fp-irreducible'', etc) if there exist $n$ and $Z_n \subset L_n^+ X(\CC)$ such that $Z = \pi_n^{-1}(Z_n)$ (respectively, ``fp-open'', ``fp-irreducible'', etc).

If $Z$ is fp-constructible, define $\overline{Z} = \pi_n^{-1}(\overline{Z_n})$. We also define $\operatorname{codim}_{L^+X}(Z) = \operatorname{codim}_{L^+_n)}(Z_n)$. 

These definitions are independent of $n$ and $Z_n$. 

For $n\geq 0$, let $L_n X$ be the closed subfunctor of $LX$ classifying maps factoring through $t^{-n}\mathcal{O}$. Then $LX = \varinjlim_n L_n X$. If $Y \subset L_n X$ for some $n$, we may import the definitions of fp-constructible, fp-irreducible and so on by post-composing $X \to t^{-n} \mathcal{O}$ with multiplication by $t^n\colon t^{-n} \mathcal{O}\to \mathcal{O}$. This is independent of the choice of $n$.

\begin{lemma}\label{lem:KL_splitting_type}
    Let $P \subset LG$, and $O_P \subset \operatorname{Lie} G_P$ a nilpotent orbit. Then there exists a unique conjugacy class $[w] \in \underline{W}$ such that $(O_P\cap \mathfrak{n}_P) + \operatorname{Lie}P^+$ contains an open subset of regular semisimple elements of type $[w]$.
\end{lemma}

\begin{proof}[Proof of \Cref{lem:KL_splitting_type}]

By \cite[Section 10.11]{jantzen2004nilpotent}, every irreducible component of $O_P \cap \mathfrak{n}_P$ has the same dimension. Moreover, if $E \subset (O_P \cap \mathfrak{n}_P)$ and $E'$ are irreducible components, the sets $E + \operatorname{Lie}P^+$ and $E' + \operatorname{Lie}P^+$ are $P$-conjugate. Therefore, it suffices to prove the statement for some $E+ \operatorname{Lie}P^+$.

Let $Y = t( E + \operatorname{Lie}P^+)$. Then $Y \subset L^+G$. Since $O_P \cap \mathfrak{n}_P$ is a Zariski constructible set in $\operatorname{Lie} G_P$, it is defined by equations $\{f_i = 0 , g_i \neq 0\}$. Then it is clear that these equations, along with equations defining $\operatorname{Lie} P^+$, define a set $Y_4 \subset \mfg(\CC[t]/(t^4))$ such that $Y = \pi_4^{-1} (Y_4)$, where $\pi_4\colon L^+\mfg \to \mfg(\CC[t]/(t^4))$ is the projection. This shows that $Y$ is fp-constructible.

Let $x \in \mathbb{X}_*(T)\otimes \RR$ be a point in the compact spherical alcove corresponding to $P$. Note that $tE$ is contained in the subspace of roots $\alpha \in \Phi$ such that $\alpha(x) = 0$ , $\alpha (x) =1$ or $\alpha(x) = -1$. Let $X'$ be the subspace of the remaining root vectors, and $X$ be the subspace of $X'$ cut out by $e_\alpha \in t\mfg$ if $\alpha \in \Phi^+$, $e_\alpha \in t^2 \mfg$ otherwise. Then $Y_4$ is the product, as schemes over $\CC$ of $tE$, $X'$ and $t^2 \mathfrak{t}$. Each of these is irreducible: $tE$ is by assumption, and $X'$ and $t^2 \mathfrak{t}$ are isomorphic to affine spaces. Therefore $Y_4$ is irreducible.

By definition, we have shown that $Y$ is fp-irreducible. Thus the lemma follows from \cite[Section 6, Corollary]{kazhdan1988fixed}.
\end{proof}

\section{The map $\widetilde{\operatorname{AV}}$}

Keep the notations from the introduction. We drop the assumption that $G$ is simply connected. Let $\pi_1(G) = \mathbb{X}_*(T)/\ZZ\Phi^\vee$ be the algebraic fundamental group of $G$.

\subsection{Definition of $\widetilde{\operatorname{AV}}$} Since $G$ is no longer assumed to be simply connected, $\Fl$ may no longer be connected. However, its connected components are all isomorphic. Moreover, if $G\to G'$ is a central isogeny, then 
\begin{center}
\begin{tikzcd}
    (\Fl_G)_{\operatorname{red}} \arrow[d] \arrow[r]&(\Fl_{G'})_{\operatorname{red}} \arrow[d]\\
    \pi_1(G)\arrow[r] & \pi_1(G')
\end{tikzcd}
\end{center}
commutes, and $(\Fl_G)^\circ = \Fl_{G^{sc}}$ where $G^{sc}$ is the simply connected form of $G$. 

The isomorphism $\widetilde{W} \cong W_{\operatorname{aff}} \rtimes \pi_1(G)$ presents $\widetilde{W}$ as a quasi-Coxeter group. Recall the Iwahori decomposition
\[
\Fl = \bigsqcup_{w \in \widetilde{W}} IwI/I = \bigsqcup_{\omega \in \pi_1(G)} \bigsqcup_{w' \in W_{\operatorname{aff}}} I(\omega, w') I /I.
\]
\begin{lemma}\label{lem:w_aff_only}
The subset $\pi SS(IC_{(\omega, w')})$ depends only on $w' \in W_{\operatorname{aff}}$.
\end{lemma}
\begin{proof}
    Follows from discussion above.
\end{proof}

\begin{lemma}\label{lem:sing-supp-constant}
The subset $\pi SS(IC_w)$ is constant on the two-sided cells of $W_{\operatorname{aff}}$.
\end{lemma}
\begin{proof}
By definition, $x \leq_{LR} y$ if and only if there is a sequence $x = x_0, \dots , x_n = y$ such that either $x_{i-1} \leq_L x_i$ or $x_{i-1} \leq_R x_i$. The left order $\leq_L$ is generated by relations of the form: $x \leq_L x'$ if and only if there exists $s \in S$ such that $C_x$ appear s in $C_s C_{x'}$. Here $S$ is the set of simple affine reflections of $W_{\operatorname{aff}}$ determined by our choice of simple roots, and $C_x$ are the Kazhdan-Lusztig basis of the Hecke algebra. By definition, $x \leq_R x'$ if $x'^{-1} \leq_L x^{-1}$.

Therefore, it is enough to prove the following: if $C_x$ appears in $C_s C_{x'}$, then $\pi SS(\operatorname{IC}_x) \subset \pi SS(\operatorname{IC}_{x'})$ (and the similar statement for the $\leq_R$ relation).

Let $x$, $s$, $x'$ be as above. From the Kazhdan-Lusztig conjectures (now a theorem), $\operatorname{IC}_x$ is a direct summand of $\operatorname{IC}_s \star \operatorname{IC}_{x'}$ (perhaps with a shift). Here $\star$ is the convolution operation on the affine flag variety. Therefore, $SS(\operatorname{IC}_x) \subset SS(\operatorname{IC}_s \star \operatorname{IC}_{x'})$. It is then enough to show that $\pi SS(\operatorname{IC}_s \star \operatorname{IC}_{x'}) \subset \pi SS(\operatorname{IC}_{x'})$. 

The convolution diagram is
\begin{center}
\begin{tikzcd}
    \operatorname{Fl} \times \operatorname{Fl}  &\arrow[l,"p",swap] LG \times \operatorname{Fl}\arrow[r,"q"] & LG \times^I \operatorname{Fl} \arrow[r,"m"]& \operatorname{Fl}.
\end{tikzcd}
\end{center}
Define $\operatorname{IC}_s \tilde\boxtimes \operatorname{IC}_{x'} = (q*)^{-1}(p^* \operatorname{IC}_s \boxtimes \operatorname{IC}_{x'})$. Then $\operatorname{IC}_s \star \operatorname{IC}_{x'} = m_*(\operatorname{IC}_s \tilde\boxtimes \operatorname{IC}_{x'})$. Consider the correspondence diagram
\begin{center}
\begin{tikzcd}
T^*(LG \times^I \operatorname{Fl} ) &\arrow[l,"m'",swap] T^* Fl \times_{Fl} (LG \times^I \operatorname{Fl}) \arrow[r,"m_\pi"] & T^*\operatorname{Fl}.
\end{tikzcd}
\end{center}
According to \cite[Proposition 5.4.1]{kashiwara_sheaves_1990}, $SS(\operatorname{IC}_s \tilde\boxtimes \operatorname{IC}_{x'}) \subset SS(\operatorname{IC}_s )\times SS(\operatorname{IC}_{x'})$, and since $m$ is proper,
\begin{align*}
\pi SS \operatorname{IC}_s \star \operatorname{IC}_{x'} \subset \pi m_\pi m'^{-1} SS(\operatorname{IC}_s )\times SS(\operatorname{IC}_{x'}).
\end{align*}
The differential $dm$ is the addition map on tangent spaces. Its dual, therefore, sends $l \in T^*\operatorname{Fl}$ to the functional $(v,w) \mapsto l(v+w)$. Under our identification of the fibers of $T^*\operatorname{Fl}$ with $\operatorname{Ad}(g)(\operatorname{Lie} I^+)$, $m'$ is the map sending $v \mapsto (v,v)$. Therefore, the equation above imples $\pi SS \operatorname{IC}_s \star \operatorname{IC}_{x'} \subset \pi SS\operatorname{IC}_{x'}$. This finishes the proof for the left order relation, and the right order relation is handled in the same way.
\end{proof}

Let $P \subset LG(\CC)$ be a parahoric subgroup. Recall that there is an exact sequence
\[
1 \to P^+ \to P \to G_P \to 1
\]
We refer to $G_P = P/P^+$ as the Levi quotient. It is a reductive group with Weyl group $W_P$.

\begin{theorem}\label{thm:AV_construction}
\begin{enumerate}
    \item The set $V_w \coloneqq \pi SS(IC_w) \subset L\mfg$ is contained in $t^{-1} L^+\mfg$. It is the union of finitely many fp-irreducible and fp-constructible subsets.
    \item There is a nonempty fp-open subset $U \subset V_w$ and a conjugacy class $[w] \in \underline{W}$ such that for any $x\in U$, $x$ is regular semisimple topologically nilpotent in a Cartan subalgebra of $L\mfg$ of type $[w]$.
    \item The class $[w]$ does not depend on the choice of $U$.
\end{enumerate}
\end{theorem}

For the notion of fp-open, fp-irreducible etc, see \Cref{sect:arc_topology}.

\begin{proof}
According to \cite{lusztig_cells_1989}, every two sided cell of $W_{\operatorname{aff}}$ intersects some finite parahoric Weyl subgroup (generated by a subset of $\Sigma_{\operatorname{aff}}$). By \Cref{lem:sing-supp-constant}, we may assume that $w$ lies in some finite parahoric subgroup $W_P$ corresponding to a standard parahoric subgroup $P \subset LG(\CC)$. Let $F_{G_P} = P/I \cong G_P / B_{G_P}$ be the classical (finite) flag variety of $G_P$. Let $i\colon F_{G,P} \hookrightarrow \Fl$ be the inclusion induced by $P \hookrightarrow LG$.

We have a diagram
\begin{center}
\begin{tikzcd}
T^*F_{G_P} \arrow[d]& T^*\operatorname{Fl}\arrow[d]\arrow[l,"di"]\\
F_{G_P}\arrow[r, hook,"i"]& \operatorname{Fl}.
\end{tikzcd}
\end{center}
Then the closure of $IwI/I$ is contained in the closed subset $P/I = F_{G_P}$. Since $IC_w$ is supported on the $P/I$, it is equal to $i_*IC_{w,P}$, where $\operatorname{IC}_{w,P}$ denotes the IC sheaf viewed in $D^b(F_{G_P})$.

By \Cref{lem:ker_di}, the kernel of $di$ is $\operatorname{Lie} P^+$ at every point. By \cite[Proposition 5.4.4]{kashiwara_sheaves_1990},
\[
V_w = \operatorname{Lie} P^+ + \operatorname{AV}_{G_P}(w) = \operatorname{Lie}P^+ + (\overline{O}\cap \mathfrak{n}_P),
\]
where $O \subset \operatorname{Lie} G_P$ is a nilpotent orbit (see \Cref{subsect:intro_classical}).

This description settles assertion (1) of the theorem. Then the rest of the assertions in the theorem follow from \Cref{lem:KL_splitting_type}.
\end{proof}

\begin{lemma}\label{lem:ker_di}
    The kernel of $di\colon T^*\Fl \to T^* F_{G_P}$ over any point $pI/I \in P/I$ is identified with $\operatorname{Lie} P^+$. 
\end{lemma}
\begin{proof}
    For a root $\alpha$ of $G$, let $\mfg_\alpha$ be the corresponding root subspace. Let $x \in \mathbb{X}_*(T) \otimes \RR$, $|x|=1$ be a point in the compact spherical alcove corresponding to $P$. We may assume that $x$ is in the fundamental Weyl chamber, so that $\alpha(x) \geq 0$ for every $\alpha \in \Phi^+$. Then $\operatorname{Lie}P$ is generated by
    \[
    \mathfrak{t}(\mathcal{O}) \text{ and }  \mfg_\alpha(t^{n_\alpha}\mathcal{O}) \text{ with }\alpha(x) + n_\alpha \geq 0.
    \]
    Note that $I$ corresponds to the point $x_I \in \mathbb{X}_*(T) \otimes \RR$ such that $\alpha(x) >0$ for all $\alpha \in \Phi^+$. 
    
    Therefore, we identify the fiber of $T^*F_{G_P}$ over $I/I$ with the subspace generated by
    \[
    \mfg_\alpha(t^{-1}\mathcal{O}) \text{ for } \alpha\in\Phi^+\text{, } \alpha(x)=1 \text{ and } \mfg_\alpha(\mathcal{O}) \text{ for } \alpha \notin\Phi^+ \text{, }\alpha(x) =0.
    \]
    The roots appearing are precisely the negative roots for $G_P$, using the pinning given by the maximal torus $T$, and Borel subgroup $I/P^+$. 

    Since the fiber of $T^*\Fl$ over $I/I$ is $\operatorname{Lie}I^+$, we conclude that the kernel of $di$ is generated by
    \[
    \mathfrak{t}(t\mathcal{O}) \text{ and }  \mfg_\alpha(t^{n_\alpha}\mathcal{O}) \text{ with }\alpha(x) + n_\alpha > 0.
    \]
    which is exactly $\operatorname{Lie} P^+$.

    Since $P^+$ is normal in $P$, the lemma follows.
\end{proof}

We now arrive at the central definition of the paper.
\begin{definition}
    Let $C \in \operatorname{Cells}(W_{\operatorname{aff}})$. Choose $w \in C$, and let $[w]$ be associated to $V_w$ as in \Cref{thm:AV_construction}. The assignment
    \[
    \widetilde{\operatorname{AV}} \colon C \mapsto [w]
    \]
    is independent of the choice of $w \in C$.
\end{definition}

\subsection{The main theorem} \label{sect:main_thm}

We first recall the definitions of the (other) maps appearing in \Cref{thm:intro_diagram}.

\subsubsection{Lusztig's map $\underline{c}$} Let $u \in G^\vee$ be unipotent, and $s \in Z_{G^\vee}(u)$. Let $\rho$ be a simple, admissible module for $A\coloneqq Z_{G^\vee}(su)/Z_{G^\vee}^\circ (su)$. Define
\[
\mathcal{K}_{u,s,\rho} = \Hom_A (\rho, K_0^{\langle s \rangle \times \mathbb{G}_m }(F_u)\otimes \CC).
\]
Here, $F_u$ is the Springer fiber of $u$, and $K_0^{\langle s \rangle \times \mathbb{G}_m }$ represents equivariant $K$-theory of coherent sheaves.

Let $H = \ZZ[v,v^{-1}][\widetilde{T_w}]$ denote the affine Hecke algebra associated to $W_{\operatorname{aff}}$. Note that this is the affine Weyl group of $G$. Let $J$ be the corresponding asymptotic Hecke algebra. It is an associative algebra over $\CC$, equipped with a map
\begin{equation}\label{eqn:H_to_J}
H \to J \otimes \ZZ[v,v^{-1}].
\end{equation}
It is a direct sum of algebras
\begin{equation}\label{eqn:J_direct_sum}
    J = \bigoplus_{c \in \operatorname{Cells}(W_{\operatorname{aff}})} J_{c}.
\end{equation}
Then $\mathcal{K}_{u,s,\rho}$ carries the structure of a $H$-module. On the other hand, any $J$ module carries the structure of a $H$-module due to \eqref{eqn:H_to_J}. In \cite[Theorem 4.2]{lusztig_cells_1989}, Lusztig demonstrates that this induces a bijection
\[
\left\{\text{$G^\vee$-conjugacy classes of triples $(u,s,\rho)$}\right\} \leftrightarrow \left\{\text{Simple $J$-modules}\right\}.
\]
Let $E$ be a simple $J$-module. Due to \eqref{eqn:J_direct_sum}, there exists a unique 2-sided cell $c$ such that $J_c$ does not act by 0 on $E$. This gives a map
\[
\left\{\text{$G^\vee$-conjugacy classes of triples $(u,s,\rho)$}\right\} \to \left\{\text{2-sided cells of $W_{\operatorname{aff}}$} \right\}.
\]
Lusztig then shows that this map depends only on the conjugacy class of $u \in G^\vee$. Let us denote the two-sided cell obtained in this manner by $\underline{c}(u)$. Lusztig then proves that $\underline{c}\colon \mathcal{N}_{G^\vee} \to \operatorname{Cells}(W_{\operatorname{aff}})$ is a bijection.

\subsubsection{The Kazhdan-Lusztig map $\operatorname{KL}_G$} Let us recall that an element $\gamma \in L\mfg$ is said to be topologically nilpotent if $\operatorname{ad}(\gamma)^r \to 0$ as $r\to \infty$ in $\End L\mfg$ (in the $t$-adic topology).

Let $e \in \mathcal{N}$. In \cite[Section 9]{kazhdan1988fixed}, Kazhdan and Lusztig show that there exists $[w] \in W$ such that a generic element in $e + L^+\mfg \subset L\mfg$ is topologically nilpotent and regular semisimple of type $[w] \in \underline{W}$. Here, generic means that there is a dense open subset of such elements. Then $\operatorname{KL}_G(e) = [w]$. This assignment only depends on the nilpotent orbit containing $e$.

More generally, for any parahoric subgroup $P$, if $e_P \in \mathcal{N}_{G_P}$, one can make the same definition (using \Cref{lem:KL_splitting_type}): there exists $[w] \in \underline{W}$ such that a generic element in the set $e_P + P^+$ is regular semisimple of type $[w]$.
\begin{lemma}\label{lem:parahoric_nilpotent}
    Any element $\widetilde{e_P} \in e_P + P^+$ is topologically nilpotent.
\end{lemma}
\begin{proof}
    By \cite[Lemma 2.1]{kazhdan1988fixed}, $\widetilde{e_P}$ is topologically nilpotent if and only if it is $LG$-conjugate to an element $\gamma \in L^+G$ such that $\gamma \pmod t \in \mfg$ is nilpotent. 

    Let $x \in \mathbb{X}_*(T)\otimes \RR$ correspond to $P$. Let $\beta$ be the highest root of $G$. Let
    \[
    S_P = \{\alpha_i \mid \alpha_i(x) = 0\} \cup (\{-\beta\} \text{ if } \beta(x) =1)\subset \Phi.
    \]
    These are the gradients of the basic affine roots vanishing at $x$. They form a set of simple roots for $G_P$. Let $\mathfrak{n}_P$ denote the nilpotent subalgebra of $\mfg_P$ determined by the positive root system corresponding to $S_P$. Then
    \[
    \mathfrak{n}_P \subset \bigoplus_{\substack{\alpha \in \Phi^+\\ \alpha(x) = 0}} \mfg_\alpha(\mathcal{O}) \bigoplus_{\substack{\alpha \in \Phi^- \\\alpha(x) = -1}} \mfg_\alpha(t \mathcal{O})
    \]
    
    By conjugating by an element of $P$, we may assume that $e_P \in \mathfrak{n}_P$. Note that $P^+ \subset I$, hence it is now clear that $\widetilde{e_P} \pmod{t} \in \mathfrak{n}_G$.  
\end{proof}
We will denote $\operatorname{KL}_G^P(e) = [w]$.

This paper is dedicated to proving
\begin{theorem}\label{thm_main}
    The diagram
    \begin{center}
    \begin{tikzcd}
    \operatorname{Cells}(W_{\operatorname{aff}})\arrow[rr, "\underline{c}" ,leftrightarrow] \arrow[dr,"\widetilde{\operatorname{AV}}"]& & \underline{\mathcal{N}_{G^\vee}}\\
    & \underline{W}.\arrow[ur, hookleftarrow , "\operatorname{KL}_{G^\vee}",sloped] &
    \end{tikzcd}
    \end{center}
    commutes.
\end{theorem}
\begin{corollary}
    The map $\widetilde{AV}$ is injective.
\end{corollary}
\begin{proof}
    This follows from \Cref{thm_main} and the fact that $\operatorname{KL}$ is injective \cite[Theorem 1.11]{yun2021minimal}.
\end{proof}

\begin{lemma}
    \begin{enumerate}
        \item Let $G \to G'$ be a finite isogeny. Then \Cref{thm_main} holds for $G$ if and only if it holds for $G'$.
        \item \Cref{thm_main} holds for $G$ if and only if it holds for every simple factor of $G$ (well-defined up to isogeny).
    \end{enumerate}
\end{lemma}
\begin{proof}
    The first statement directly follows from \Cref{lem:w_aff_only} and \cite[4.10]{lusztig_cells_1989}.
    
    For the second, $G$ is isogenous to a product $A \times \prod_i G_i$ where $G_i$ are simple and $A$ is a torus. The theorem is vacuous for tori, and clearly holds for the product if it is true for each $G_i$. 
\end{proof}

In the rest of the paper, we therefore assume without loss of generality that $G$ is almost simple and simply connected.

\section{An identity of Kazhdan-Lusztig maps}\label{sect:reformulation}

In this section, we give a reformulation \Cref{thm_main} not involving the map $\widetilde{AV}$. First we recall a few more definitions due to Lusztig.

Let $E$ be an irreducible representation of $W$. To $E$, Lusztig \cite{lusztig1979class} attaches two integers $a_{E}$ and $b_{E}$, as follows. Let $D_E(q)$ be the generic degree polynomial associated to $E$. Then $D_E(q)= \gamma_E q^{a_E} + \text{ higher powers}$. The integer $b_E$ is defined to be the smallest nonnegative integer $i$ such that $E$ appears in the $i$th symmetric power of the reflection representation of $W$.

Let $O \subset \mfg$ be a nilpotent orbit. Let $E_O$ denote the corresponding irreducible representation of $W$ under the Springer correspondence. Lusztig observes \cite[2.1]{lusztig1979class} that $a_{E_O} \leq b_{E_O}$. When this is an equality, $E_O$ is called \textit{special}. The corresponding nilpotent orbit $O$ is also called \textit{special}.

Let $W_P \subset W_{\operatorname{aff}}$ be a parahoric subgroup. By definition, it is a finite reflection subgroup formed by a proper subset of $\Sigma_{\operatorname{aff}}$. The projection map $W_{\operatorname{aff}} \to W$ allows us to identify $W_P$ with its image in $W$, which we shall also denote by $W_P$ when the context is clear.

Let $E_P$ be a $W_P$-module. Let $E = \operatorname{Ind}_{W_P}^W E_P$. Then it is known (due to Macdonald-Lusztig-Spaltenstein, \cite[Theorem 5.2.6]{geck2000characters}) that for any $E'$ appearing in $E$, $b_{E'} \geq b_{E_P}$, and moreover, equality is achieved for a unique $E''$ (appearing with multiplicity 1). This theorem is attributed to Macdonald-Lusztig-Spaltenstein. We shall denote $j_{W_P}^W E_P \coloneqq E''$. 

If $E_P$ is a special $W_P$-module, Lusztig \cite[Theorem 1.5]{LUSZTIG20093418} shows that $j_{W_P}^W E_P$ is in the image of the Springer correspondence for $G^\vee$.

Recall that $\operatorname{Spr}$ refers to the Springer correspondence for the reductive group in the subscript. The secondary main theorem of this paper is
\begin{theorem}\label{thm_KL}
     For any parahoric subgroup $P \subset LG(\CC)$, and any special representation $E_P$ of $W_P$,
    \[
    \operatorname{KL}_{G^\vee} \left(\operatorname{Spr}_{G^\vee} j_{W_P}^W E_P\right) = \operatorname{KL}_G^P \left(\operatorname{Spr}_{G_P} E_P\right).
    \]
\end{theorem}

\begin{proposition}\label{prop:equivalence}
    \Cref{thm_main} holds if and only if \Cref{thm_KL} holds.
\end{proposition}
\begin{proof}[Proof of \Cref{prop:equivalence}]
    Recall the notations in the proof of \Cref{thm:AV_construction}. Let $c$ be a two-sided cell of $W_{\operatorname{aff}}$, and $P$ a parahoric subgroup so that $W_P \cap W_{\operatorname{aff}}$ is nonempty. Let $w \in W_P \cap W_{\operatorname{aff}}$. Then $\widetilde{\operatorname{AV}}(c)$ is defined to be the Cartan type of a generic element in $V_w = \operatorname{Lie} P^+ + (\overline{O}\cap \mathfrak{n}_P)$. Since every irreducible component of $\overline{O} \cap \mathfrak{n}_P$ has the same dimension and intersects $O$ (see \cite[Section 10.11]{jantzen2004nilpotent}), the Cartan type of a generic element agrees with $\operatorname{KL}_G^P(O)$.

    On the other hand, it is known that $O$ is the special nilpotent orbit corresponding (by the Springer correspondence) to the special $W_P$-representation attached (in the sense of Lusztig) to the two-sided cell $c \cap W_P$ of $W_P$. In other words, writing $E_P$ for the $W_P$ representation attached to $c \cap W_P$, we have
    \[
    \widetilde{\operatorname{AV}}(c) = \operatorname{KL}_G^P \left(\operatorname{Spr}_{G_P} E_P\right).
    \]

    On the other hand, it is known \cite[6.5, 6.7]{lusztig_cells_1989} that the nilpotent orbit $\underline{c}^{-1}(c) \eqqcolon O^\vee \subset \mfg^\vee$ corresponds to the Springer representation $j_{W_P}^W E_P$. This is exactly the left hand side of the equation in \Cref{thm_KL}.
\end{proof}

\subsection{Geometric interpretation of $j$-induction}  

When $W_P$ is a reflection subgroup of $W$, i.e. the Weyl group of a Levi subgroup of $G$, Lusztig and Spaltenstein give a geometric description of truncated induction \cite{lusztig1979induced}, which we now recall.

Let $L \subset P \subset G$ be a Levi and a parabolic subgroup of $G$, respectively. Let $\mathfrak{u}_P$ be the unipotent radical of $\operatorname{Lie} P$. Let $O_L \subset \mathfrak{l}$ be a nilpotent orbit. Lusztig and Spaltenstein show that there is a unique nilpotent orbit $O \subset \mfg$ such that $O \cap (O_L + \mathfrak{u}_P)$ is dense in $O_L + \mathfrak{u}_P$. Moreover, $O$ is independent of the choice of $P$. Then $O = \operatorname{Ind}_L^G(O_L)$ is called the induced nilpotent orbit. Let $E_L$ correspond to $O_L$ under the Springer correspondence. They then show that $O$ corresponds to $j_{W_L}^W(E_L)$ under the Springer correspondence.

Note that this construction only works when $L = Z_G(s)$, $s \in \mfg$ semisimple. We now offer a generalization (restricted to special nilpotent orbits) to the ``pseudo-Levi'' case, $L = Z_G(s)$, $s \in G$ semisimple. Recall that Yun \cite[Theorem 1.11]{yun2021minimal} has defined a map $\operatorname{RT_{min,G^\vee}}\colon \underline{W} \to \underline{\mathcal{N}_{G^\vee}}$. The definition involves the geometry of affine Springer fibers and various related arc spaces. Moreover, he shows that $\operatorname{RT_{min,G^\vee}}$ is a section of $\operatorname{KL}_{G^\vee}$.

The following geometric interpretation of Lusztig's $j$-induction follows immediately from \Cref{thm_KL}.
\begin{corollary}
    The diagram
    \begin{center}
    \begin{tikzcd}
        \underline{\mathcal{N}_{G_P}}^{sp} \arrow[d, "\operatorname{Spr}", leftrightarrow] \arrow[rr, "\operatorname{RT_{min,G^\vee}}\operatorname{KL}_G^P"]& &\underline{\mathcal{N}_{G^\vee}} \arrow[d,"\operatorname{Spr}", leftrightarrow]\\
        \operatorname{Rep}^{sp}(W_P) \arrow[rr, "j_{W_P}^W"]& & \operatorname{Rep}(W)
    \end{tikzcd}
    \end{center}
    commutes.
\end{corollary}

The values of the Kazhdan-Lusztig maps $\operatorname{KL}_G$ have been computed by Spaltenstein and Yun in \cite{yun2021minimal,spaltenstein1988polynomials,spaltenstein1990evenorth,spaltenstein1990exceptional}. The parahoric Kazhdan-Lusztig maps $\operatorname{KL}_G^P$ have only been computed for $G$ of classical type by Spaltenstein \cite[Theorem 8.11]{spaltenstein1992order}. To the best of our knowledge, \Cref{thm_KL} gives the first computation of the parahoric Kazhdan-Lusztig maps for exceptional types (restricted to special nilpotent orbits).

\section{Affine Springer fibers}

In this section, we use topological properties of $\mathfrak{c}$ to reduce the proof of \Cref{thm_KL} to studying properties of affine Springer fibers (\Cref{prop:springer_fiber_imply_thm}).

\subsection{Subsets of $L^\heartsuit \mfc$} We recall definitions and results from \cite[Section 5]{yun2021minimal}.

Recall the discriminant function $\Delta\colon L\mfc \to L\mathbb{A}^1$ (see Equation \eqref{eqn:discriminant}). By definition, $L^\heartsuit \mfc$ is the complement of the vanishing set of the discriminant function. Thus, it is an open subscheme of the closed subscheme $L^{++}\mfc \subset L^+\mfc$ of topologically nilpotent elements.

For $n >0$, let $(L^+\mfc)_{\leq n} \subset L^+\mfc$ be the locus such that $\operatorname{val}\Delta(\gamma) \leq n$ for all $\gamma \in (L^+\mfc)_{\leq n}$. Let $(L^\heartsuit \mfc)_{\leq n} = L^\heartsuit \mfc \cap (L^+\mfc)_{\leq n}$. Clearly,
\[
L^\heartsuit \mfc = \bigcup_{n\geq 0} (L^\heartsuit \mfc)_{\leq n}.
\]
and each $(L^\heartsuit \mfc)_{\leq n}$ is fp-open in $L^{++} \mfc$.

We say that $Z \subset L^\heartsuit \mfc$ is fp-closed (resp. fp-open, etc) if $Z \cap (L^\heartsuit \mfc)_{\leq n}$ is fp-closed (resp. fp-open, etc) as a subset of $L^+\mfc$ for any $n\geq 0$. If $Z \subset L^\heartsuit \mfc$ is fp-constructible, we define
\[
\operatorname{codim}_{L^\heartsuit\mfc}(Z) = \min_{n \geq 1} \operatorname{codim}_{(L^\heartsuit\mfc)_{\leq n}}(Z \cap (L^\heartsuit\mfc)_{\leq n}).
\]

The key lemma justifying these definitions is
\begin{lemma}[ {\cite[Lemma 5.2]{yun2021minimal}}]\label{lem:c_constructible}
    Let $n \geq 1$. Suppose that $Z \subset \chi^{-1}(L^+\mfc)_{\leq n} \subset L^+ \mfg$ is fp-constructible. Then so is $\chi(Z) \subset (L^+\mfc)_{\leq n}$.
\end{lemma}

Let $O\subset \mfg$ be a nilpotent orbit. Yun shows in \cite[Lemma 5.3]{yun2021minimal} that $(L^\heartsuit \mfc)_{\overline{O}}$ is closed in $L^\heartsuit \mfc$. By a very slight modification of the proof, we have:

\begin{lemma}\label{lem:c_o_closed}
    Let $P \subset LG(\CC)$ be a parahoric subgroup, and $O_P \subset \mfg_P$ be a nilpotent orbit of $G_P$. Then $(L^\heartsuit \mfc)_{\overline{O_P}}^P$ is fp-closed in $L^\heartsuit \mfc$.
\end{lemma}

For clarity, we reproduce the proof with relevant modifications here.

\begin{proof}
    By \Cref{lem:c_constructible}, $(L^\heartsuit \mfc)_{\overline{O_P}}^P$ is constructible. Hence there exists $N,n \in \mathbb{N}$, $Y \subset (L^\heartsuit\mfc)_{\leq n}$ such that $(L^\heartsuit \mfc)_{\overline{O_P}}^P = \pi_N^{-1}(Y)$.

    Let $V$, $\overline{V}$ be lifts of $Y$, $\overline{Y}$ respectively to $L^\heartsuit \mfg$, by first the zero section to $(L^\heartsuit \mfc)_{\leq n}$ and then the Kostant section. It suffices to prove: any $v \in \overline{V}$ is $LG$-conjugate to a point in $\overline{O_P} + P^+$. For this would show that $\pi_N^{-1}(\overline{Y}) \subset (L^\heartsuit \mfc)_{\overline{O_P}}^P$.
    
    Let $\Fl^P$, $\Fl^P_\gamma$ denote the partial affine flag variety of type $P$, and the corresponding affine Springer fiber respectively. We have the evaluation morphism $\operatorname{ev}_{P\gamma} \colon \Fl_\gamma^P \to [\mfg_P / G_P]$. For more details, see \Cref{sect:affine_springer_intro}.
    
    Define
    \[
    X_{\overline{V}}^{\overline{O_P}} = \left \{ \Fl_\gamma^P \times \gamma \bigm | \gamma \in \overline{V} \text{, } \operatorname{ev}_{P,\gamma}^{-1} (\overline{O_P}/G) \neq \varnothing\right\} \subset \Fl^P \times \overline{V}.
    \]
    The projection $p \colon \Fl^P \times \overline{V} \to \overline{V}$ is ind-proper. Since $X_{\overline{V}}^{\overline{O_P}}$ is closed, its image under $p$ is a countable union of closed subsets. Moreover, this image contains $V$, hence equals $\overline{V}$. This was what we wanted.
\end{proof}

Let $d_{O_P}$ be the dimension of the Springer fiber (of $G_P$) at any element in $O_P$.

The following theorem is a straightforward generalization of \cite[Theorem 6.1]{yun2021minimal}.
\begin{proposition}\label{prop:closure}
    Let $P \subset LG(\CC)$ be a parahoric subgroup. Let $O \subset \mfg$ be a nilpotent orbit and $[w] = \operatorname{KL}_G^P(O)$. Then
    \[
    (L^\heartsuit\mfc)_{\overline{O_P}}^P = \overline{(L^\heartsuit\mfc)^{sh}_{[w]}} \cap L^\heartsuit \mfc
    \]
    and
    \[
    \operatorname{codim}_{L^\heartsuit \mfc} (L^\heartsuit\mfc)_{\overline{O_P}}^P = d_{O_P}.
    \]
\end{proposition}
\begin{proof}
    By \cite[Proposition 8.2]{kazhdan1988fixed}, $\delta_{[w]} = d_{O_P}$. We claim that
    \[
    \operatorname{codim}_{L^\heartsuit \mfc}(L^\heartsuit \mfc)_{\overline{O_P}}^P \leq d_{O_P}.
    \]
    Let $\mathfrak{I}$ be the lie algebra of $I$, and $(L^\heartsuit\mathfrak{I} )_{\overline{O_P}}^P = \chi^{-1} (L^\heartsuit \mfc)_{\overline{O_P}}^P$. In loc. cit, it is shown that
    \[
    \operatorname{codim}_{L^\heartsuit \mfc}(L^\heartsuit \mfc)_{\overline{O_P}}^P = \operatorname{codim}_{L^\heartsuit \mathfrak{I}}(L^\heartsuit\mathfrak{I} )_{\overline{O_P}}^P.
    \]
    Let $\mathfrak{n}_P$ be the nilpotent subalgebra of $\mfg_P$ as in \Cref{lem:parahoric_nilpotent}. To bound the right hand side, note that
    \[
    L^\heartsuit \mathfrak{I} = (\mathfrak{n} + tL^+\mfg) \cap L^\heartsuit \mfg =(\mathfrak{n}_P + (\mathfrak{n} \setminus \mathfrak{n}_P) + tL^+ \mfg) \cap L^\heartsuit \mfg = (\mathfrak{n}_P + P^+) \cap L^\heartsuit\mfg.
    \]
    Since $(L^\heartsuit\mathfrak{I} )_{\overline{O_P}}^P$ contains $(O_P \cap \mathfrak{n}_P) + P^+$, and $O_P \cap \mathfrak{n}_P$ has codimension $d_{O_P}$ in $\mathfrak{n}_P$, the inequality follows.
    
    The rest of the proof is identical to the one in \cite[Theorem 6.1]{yun2021minimal}.
\end{proof}

\subsection{Affine Springer fibers}\label{sect:affine_springer_intro}

Let $\gamma \in L\mfg$. For any parahoric subgroup $P \subset LG(\CC)$, consider the subfunctor $Y \subset \Fl^P$ given by
\[
Y(R) = \{gP \in \Fl_P(R) \mid \operatorname{Ad}(g^{-1})\gamma \in R\otimes \operatorname{Lie} P \}.
\]
Recall that the affine Springer fiber of type $P$ is defined to be the reduced ind-subscheme of $Y$. We denote it by $\Fl^P_\gamma$.

It is know that $\Fl_\gamma^P$ is finite dimensional if and only if $\gamma$ is regular semisimple. In this case, it is known that
\[
\dim \Fl_\gamma = \delta(\gamma)
\]
(see \eqref{eqn:springer_dimension}). This is the dimension formula of Bezrukavnikov \cite{bezrukavnikov1996dimension}.

Let $\gamma \in L^\heartsuit\mfc$. Since the global dimension of $\CC((t))$ is at most 1, the fiber $\chi^{-1}(\gamma) \subset L\mfg$ is a single $LG$-conjugacy class. For $\gamma_1$, $\gamma_2 \in \chi^{-1}(\gamma)$ with $\gamma_1 = \operatorname{Ad}(g) \gamma_2$, left multiplication by $g$ gives an isomorphism $\Fl_{\gamma_1}^P \to \Fl_{\gamma_2}^P$.

Let $\kappa\colon \mfc \to \mfg$ be the Kostant section. This similarly extends to $\kappa \colon L^\heartsuit \mfc \to L^\heartsuit \mfg$. Let $\gamma' \in L\mfg$, and set $\gamma = \chi(\gamma')$. By the preceding paragraph, the isomorphism class of $\Fl_{\gamma'}^P$ depends only on $\gamma$. In particular, it is isomorphic to $\Fl_{\kappa(\gamma)}^P$. We sometimes abuse notation and write $\Fl_{\gamma}^P$ when no confusion will occur.

\subsubsection{Reduction type} For any $P$, $\gamma$ as above, we have the evaluation morphism
\[
\operatorname{ev}_{P\gamma} \colon \Fl_\gamma^P \to [\mfg_P / G_P]
\]
sending
\[
gP \mapsto \operatorname{Ad}(g^{-1})\gamma  \pmod{P^+} \in \mfg_P / G_P.
\]
Let $O_P$ be a nilpotent orbit in $\mfg_P$. We say that $\Fl_\gamma^P$ has \textit{reduction type} $O_P$ if $\operatorname{ev}_P ^{-1} (O_P) \neq \varnothing$. When $P=I$, we abbreviate notation by dropping the superscript $P$ if there is no ambiguity.

\subsection{The dual torus} Fix a $W$-equivariant isomorphism $\mathfrak{t} \cong \mathfrak{t^*}$. This allows us to identify the Chevalley base of $G$ and $G^\vee$:
\[
\mfg // G  \cong \mathfrak{t}//W \cong \mathfrak{t^*}//W \cong \mfg^\vee//G^\vee = \mfc.
\]
\begin{lemma}\label{lem:root_val}
    Let $\gamma \in L\mathfrak{t}$ correspond to $\gamma^\vee \in L\mathfrak{t}^\vee = L\mathfrak{t}^*$. Let $\alpha$ be a root of $G$, and $\alpha^\vee$ be its coroot, which we view as a root of $G^\vee$. Then the root valuations $\operatorname{val}_t\alpha(\gamma)$ and $\operatorname{val}_t \alpha^\vee(\gamma^\vee)$ are equal.
\end{lemma}
\begin{proof}
    Let $( - | - )$ be the positive definite symmetric bilinear form induced by the identification $\mathfrak{t} \cong \mathfrak{t}^*$. Then $\alpha^\vee \in \mathfrak{t}^*$ is identified with $\frac{2a}{(a|a)} \in \mathfrak{t}$. It follows that
    \[
    \alpha^\vee (\gamma^\vee) = \frac{2}{(a|a)} \alpha(\gamma),
    \]
    i.e. they differ by a nonzero complex scalar. Therefore their valuations are equal.
\end{proof}
Denote the affine Springer fiber (with respect to the dual Iwahori) with respect to $G^\vee$ (with the dual pinning induced from $G$) by $\Fl_\gamma^\vee$.
\begin{lemma}\label{lem:fl_dual_same_dim}
    We have $\dim \Fl_\gamma^\vee = \dim \Fl_\gamma$.
\end{lemma}
\begin{proof}
    Follows directly from Bezrukavnikov's dimension formula and \Cref{lem:root_val}.
\end{proof}
Let $(L^\heartsuit\mfc)_{[w]}^\vee$ be the counterpart of $(L^\heartsuit\mfc)_{[w]}$ defined for $G^\vee$.
\begin{lemma}
    We have $(L^\heartsuit\mfc)_{[w]}^\vee = (L^\heartsuit\mfc)_{[w]}$.
\end{lemma}
\begin{proof}
    Let $\tau$ be the topological generator of the absolute Galois group of $\CC((t))$ that multiplies $t^{1/m}$ by $\exp(2\pi i/m)$. Let $\gamma \in L\mfc$. Then $\gamma$ has type $[w]$ if $w\tau \gamma = \gamma$, where $\tau$ acts on powers of $t$ and $w$ acts on coefficients of $\gamma$. The lemma is clear from this description.
\end{proof}
We thus suppress the superscript $\vee$ when referring to such subsets of $L^\heartsuit\mfc$. Recall that the topological notions of subsets of $L^\heartsuit\mfc$ we defined using the ind-topology of $\mathfrak{t}$, and the discriminant function. By \Cref{lem:root_val}, these topological notions are the same.

Let $O_P \subset \mfg$ be a special nilpotent orbit of $G_P$, and $E_P$ be the corresponding $W_P$ representation under the Springer correspondence for $G_P$. Let $O \subset \mfg$ be the nilpotent orbit for $G$ corresponding to $j_{W_P}^W E_P$ (notations as in \Cref{sect:reformulation}).
\begin{proposition}\label{prop:springer_fiber_imply_thm}
    Suppose that there exists a dense constructible subset $U \subset (L^\heartsuit\mfc)^P_{O_P} $ such that for any $\gamma \in U$, $\Fl_\gamma^{G^\vee}$ has reduction type $O$. Then \Cref{thm_KL} holds.
\end{proposition}
\begin{proof}

    Let $[w] = \operatorname{KL}_{G^\vee}(O)$ and $[w]_P = \operatorname{KL}_G^P(O_P)$. We wish to show that $[w] = [w]_P$. 
    
    Since $\Fl_\gamma^{G^\vee}$ has reduction type $O$, $\gamma$ is $LG$-conjugate into $O + tL^+G^\vee$. It follows that $\gamma \in (L^\heartsuit \mfc)_{O}^{G^\vee}$. 

    Note that $(L^\heartsuit \mfc)_{O}^{G^\vee}$ is fp-open in $(L^\heartsuit \mfc)_{\overline{O}}^{G^\vee}$. Indeed, it is the complement of finitely many $(L^\heartsuit\mfc)_{O'}^{G^\vee}$, $O' < O$, and each of these subsets is closed by \Cref{lem:c_o_closed}. Moreover, $(L^\heartsuit \mfc)_{O}^{G^\vee}$ is an open subset of $\chi(\overline{O}+tL^+G^\vee)$, hence irreducible (see \Cref{lem:KL_splitting_type}). Similar remarks hold for $(L^\heartsuit\mfc)_{O_P}^P$.

    It follows that
    \[
    \overline{(L^\heartsuit \mfc)_{O}^{G^\vee}} = \overline{(L^\heartsuit\mfc)_{[w]}} \cap L^\heartsuit\mfc \supset \overline{U} = \overline{(L^\heartsuit\mfc)_{O_P}^P} 
 = \overline{(L^\heartsuit\mfc)_{[w]_P}}\cap L^\heartsuit\mfc 
    \]
    where we applied \Cref{prop:closure} twice. By the same proposition,
    \[
    \operatorname{codim}_{(L^\heartsuit \mfc)}(L^\heartsuit \mfc)_{[w]} = \delta_{[w]} = d_O = b_{j_{W_P}^W E_P} = b_{E_P} = d_{O_P} = \operatorname{codim}_{(L^\heartsuit \mfc)}(L^\heartsuit \mfc)_{[w]_P}.
    \]
    If $[w] \neq [w]_P$, then
    \[
    (L^\heartsuit \mfc)_{[w]_P} \subset \overline{(L^\heartsuit \mfc)_{[w]} }  \setminus (L^\heartsuit\mfc)_{[w]}
    \]
    and the latter has codimension strictly greater than $\delta_[w]$, which is a contradiction. 
\end{proof}

\section{Affine Springer representations}

To show that the condition of \Cref{prop:springer_fiber_imply_thm} holds, we make a detailed study of Springer representations appearing in the cohomology of affine Springer fibers.

\subsection{Local Picard}\label{sect:local_picard} The material in this subsection follows \cite[Section 3.3]{ngo2010lemme}. Let $J \to \mfc$ be the regular centralizer group scheme. Let $\gamma \in L^\heartsuit \mfc$. View $\gamma$ as a map $\operatorname{Spec} \mathcal{O} \to \mfc$, and denote $J_\gamma \coloneqq \gamma^* J $. Let $\mathcal{P}_\gamma$ be the affine Grassmanian of $J_\gamma$; in other words, it associates to a $\CC$-scheme $S$ the groupoid of pairs consisting of a torsor over $S[[t]]$ along with a trivialization over $S((t))$. Its $\CC$-points are identified with $J_\gamma(F)/J_\gamma(\mathcal{O})$. 

It is known that $\mathcal{P}_\gamma$ is representable by a group ind-scheme. Let $\Lambda$ be the maximal free quotient of $\pi_0(\mathcal{P}_\gamma)$.

For any parahoric $P$, the centralizer $Z_{LG}(\gamma)$ acts on $\Fl_\gamma^P$ by left multiplication, and its action factors through $\mathcal{P}_\gamma$. Choosing an arbitrary splitting $\Lambda \to \mathcal{P}_\gamma$, we get an action of $\Lambda$ on $\Fl_\gamma^P$. It is known that $\Lambda \setminus \Fl_\gamma^P $ is a projective variety \cite[Proposition 3.4.1]{ngo2010lemme}. 

\subsection{The $\widetilde{W}$ action on cohomology}\label{sect:springer_action} In \cite[Section 5.4]{lusztig1996affine}, Lusztig constructs an action of $W_{\operatorname{aff}}$ on $H^*_c (\Fl_\gamma)$. Yun extends this to an action of $\widetilde{W}$ in \cite[Theorem 2.5]{yun2014spherical}. We briefly recall the construction, following \cite{yun2014spherical}.

We have a Cartesian diagram
\begin{center}
\begin{tikzcd}
    \Fl_\gamma \arrow[r,"\operatorname{ev}"] \arrow[d,"\pi_P^I"]& \left[\widetilde{\mfg_P}/G_P\right]\arrow[d,"\pi_{\mfg_P}"] \\
    \Fl_\gamma^P\arrow[r,"\operatorname{ev}_P"] & \left[\mfg_P/G_P\right]
\end{tikzcd}
\end{center}
where $\widetilde{\mfg_P} \to \mfg_P$ is the Grothendieck simultaneous resolution. Applying proper base change,
\[
\operatorname{ev}_P^*\pi_{\mfg_P,!} \CC  = \pi_{P,!}^I \operatorname{ev}^*\CC = \pi_{P,!}^I \CC. 
\]
By classical Springer theory, $W_P$ acts on $\pi_{\mfg_P,!} \CC$, hence it acts on the right hand side too. This gives a $W_P$ action on $H^*_c (\Fl_\gamma) = \mathbf R\Gamma_c(\pi_{P,!}^I\CC)$. Lusztig shows that these actions glue to give a $W_{\operatorname{aff}}$ action on $H^*_c(\Fl_\gamma)$.

Yun upgrades this to a $\widetilde{W}$ action as follows. Let $\Omega = X_*(T) / \ZZ \Phi^\vee$. The choice of standard Iwahori provides a splitting $\Omega \hookrightarrow \widetilde W$ and an identification $\Omega = N_{LG}(I)/I$. Then $\omega \in \Omega$ acts on $\Fl_\gamma$ by right multiplication, and hence on $H^*_c(\Fl_\gamma)$. This $\Omega$ action and the $W_{\operatorname{aff}}$ action are compatible and generate the $\widetilde{W}$ action on $H_c^*(\Fl_\gamma)$.

Let $\Lambda$ be as in \Cref{sect:local_picard}. The evaluation morphism $\Fl^P_\gamma \to [\mfg_P/G_P]$ factors through the quotient $\Lambda \setminus\Fl_\gamma^P$ because $\Lambda$ inherits its action from $Z_{LG}(\gamma)$. To spell it out: $\operatorname{Ad}( (hgI)^{-1}) \gamma = \operatorname{Ad} ((gI)^{-1})\gamma $ so the action of $Z_{LG}(\gamma)$ does not change the reduction type. In addition, since $\Omega$ acts on $\Fl_\gamma$ by right multiplication, it commutes with the $Z_{LG}(\gamma)$ action. Therefore the constructions of Lusztig and Yun extend to give a $\widetilde{W}$ action on $H^*_c(\Lambda\setminus\Fl_\gamma) = H^*(\Lambda\setminus\Fl_\gamma)$.

\subsection{The $Z_{LG}(\gamma)$ invariants} Starting from this section, we abbreviate $Z_{LG}(\gamma) = Z(\gamma)$.

Let $H^*(\Fl_\gamma/\Lambda)^{Z(\gamma)}$ denote the $Z(\gamma)$-invariant part of $H^*\left(\Fl_\gamma\big/\Lambda\right)$.
\begin{proposition}\label{prop:centralizer_invariant}
    Let $\gamma \in (L^\heartsuit \mfc)_{O_P}^P$, and assume $d_{O_P} = \delta(\gamma) = \dim \Fl_\gamma$. Let $E_{O_P}$ be the $W_P$ representation corresponding to $O_P$. Then
    \[
    \left[E_{O_P} \colon H^*\left(\Lambda\setminus\Fl_\gamma\right)^{Z(\gamma)} \right]\neq 0.
    \]
\end{proposition}
\begin{proof}
Let $\Fl_\gamma^{P,O_P} = \operatorname{ev}_P^{-1} (O_P)$. Then $Z(\gamma)$ acts on $\Lambda\setminus\Fl_\gamma^{P,O_P}$ and $W_P$ acts on $H^*_c\left(\Lambda\setminus\Fl_\gamma^{P,O_P}\right)$ by the constructions in \Cref{sect:springer_action}. Let $e_P \in O_P$. There is a fibration
\[
\left(G_P / B_P\right)_{e_P} \to \Lambda\setminus\Fl_\gamma^{O_P} \to \Lambda\setminus\Fl_\gamma^{P,O_P} 
\]
where $\Fl_\gamma^{O_P} = (\pi_P^I)^{-1} \Fl_\gamma^{P,O_P}$. By our assumption on $\delta(\gamma)$, $\Lambda\setminus\Fl_\gamma^{P,O_P}$ is a finite discrete set of points. Then
\[
H^*(\Lambda\setminus\Fl_\gamma^{O_P} )  = \bigoplus_{g\in \Lambda\setminus\Fl_\gamma^{P,O_P} } gH^*( (G_P/B_P )_{e_P}).
\]
Let $S= \{s_1, \dots ,s_n\}$ be the orbit of the unit coset $P / P^+ \in \Lambda\setminus\Fl_\gamma^{P,O_P} $, and $Z' = Z(\gamma)\cap P \subset Z(\gamma)$ be its stabilizer. Clearly, the image of $Z'$ in $P/P^+ = G_P$ is contained in $Z_{G_P}(e_P)$. By definition, the induced action of $Z_{G_P}$ on $E_{O_P}$ in $H^*((G_P/B_P)_{e_p})$ is trivial. Now, since $\gamma$ is regular semisimple, $Z(\gamma)$ is a torus, in particular abelian. Therefore, $Z'$ acts on each $s_i H^*((G_P/B_P)_{e_P})$ and acts trivially on the copy of $E_{O_P}$ contained within. Consider the subspace $V = \bigoplus_{s_i \in S} s_iE_{O_P}$. The action of $Z(\gamma)$ descends to an action on $V$, and we just showed that $Z'$ acts trivially. Note now that $Z(\gamma)/ Z'$ acts trivially on the diagonal copy of $E_{O_P}$ in $V$. We have shown that
\begin{equation}\label{eqn:Zgamma_inv}
    \left[E_{O_P} \colon H^*\left(\Lambda\setminus\Fl_\gamma^{O_P}\right)^{Z(\gamma)} \right]\neq 0.
\end{equation}
Now, we claim that $\Lambda\setminus\Fl_\gamma^{O_P}$ is closed in $\Lambda\setminus \Fl_\gamma$. Indeed, the image of $\operatorname{ev}_P \colon \Fl_\gamma^P \to \mfg_P$ does not intersect any $O' \subset \overline{O_P} \setminus O_P$, for if it did, $\Fl_\gamma^{O'} \subset \Fl_\gamma$ would have dimension strictly larger than $d_{O_P}$. Therefore $\Lambda \setminus \Fl_\gamma^{O_P} = (\pi_P^I)^{-1} \operatorname{ev}_P^{-1} (\overline{O_P})$ is closed.

Let $i \colon \Lambda\setminus \Fl_\gamma^{O_P} \to \Lambda\setminus \Fl_\gamma$ be the closed inclusion, and $j \colon U \to \Fl_\gamma$ be the complementary open inclusion. Adjunction gives an exact triangle
\[
j_! \CC_U = j_! j^! \CC \to \CC \to i_* i^* \CC = \CC_{\Lambda\setminus\Fl_\gamma^{O_P}}
\]
in $D^b_c(\Fl_\gamma)$. Applying $p_* = p_! \colon \Lambda\setminus \Fl_\gamma \to \operatorname{pt}$, we get
\[
H^*_c(U) \to H^*(\Lambda\setminus \Fl) \to H^*(\Lambda\setminus \Fl_\gamma^{O_P})
\]
It is clear from the constructions that the maps in the resulting long exact sequence are $W_P$-equivariant. 

We claim that the $W_P$ representation $E_{O_P}$ does not appear in $\Gamma_c(U)$. Granted this, the map between $E_{O_P}$-isotypic components of $H^*(\Lambda\setminus\Fl)$ and $H^*(\Lambda\setminus\Fl_\gamma^{O_P})$ is an isomorphism. Since $Z(\gamma)$ does not change the reduction type, it follows that the maps are $Z(\gamma)$ equivariant. Therefore the proposition follows from \eqref{eqn:Zgamma_inv}.

It remains to prove the last claim. To this end, we first show by induction that $W_P$ representations appearing in $H^*_c(U)$ appear in $H^*((G_P/B_P)_{e'})$ for various $e' \in \operatorname{ev}_P (\Fl_\gamma^P) \setminus O_P$. Let $O'$ be a minimal nilpotent intersecting $\operatorname{ev}_P (\Fl_\gamma^P) \setminus O_P$, and let $e' \in O'$. We have a fibration
\[
(G_P/B_P)_{e'} \to \Lambda\setminus \Fl_\gamma^{O'} \to \Lambda\setminus \Fl_\gamma^{P,O'}.
\]
Since the affine Springer action is constructed using the classical Springer action, the Leray-Serre spectral sequence
\[
E_2^{p,q} = H^p_c\left(\Lambda\setminus \Fl_\gamma^{P,O'} , H^q_c((G_P/B_P)_{e'})\right) \Rightarrow H_c^{p+q}(\Lambda\setminus \Fl_\gamma^{O'})
\]
is $W_P$-equivariant, with $W_P$ acting trivially on the cohomology of the base. Since Springer fibers are projective, the $W_P$ representations appearing in $H_c^{p+q}(\Lambda\setminus \Fl_\gamma^{O_P})$ appear in $H^*((G_P/B_P)_{e'})$. Now, since $O'$ was chosen to be minimal, $\Lambda \setminus \Fl_\gamma^{O'}$ is closed in $U$ (by arguments similar to the ones used before). Let $U'$ be the complementary open subset. Applying $p_! \colon U \to \operatorname{pt}$ to the open-closed exact triangle, we have a $W_P$ equivariant triangle
\[
H^*_c(U') \to H^*_c(U) \to H^*_c(\Lambda\setminus \Fl_\gamma^{O'})
\]
which finishes the inductive step. Then the claim follows from \Cref{lem:springer_smaller_orbit}. 
\end{proof}

\begin{lemma}\label{lem:springer_smaller_orbit}
    Let $O \subset \mfg$ be a nilpotent orbit, and $E_O$ be a representation attached to $O$ under the Springer correspondence. Then $E_O$ only appears in nilpotent elements contained in $\overline{O}$.
\end{lemma}
\begin{proof}
    Let $\pi \colon \widetilde\mfg \to \mfg$ be the Grothendieck simultaneous resolution. We have a decomposition
    \[
    \pi_*\CC[\dim \mfg ] = \bigoplus_{\mathcal{L}, O} \operatorname{IC}(O,\mathcal{L}) \otimes V_{O, \mathcal{L}}
    \]
    where $O$ are nilpotent orbits and $\mathcal{L}$ are irreducible local sytems on $O$. By classical Springer theory, $V_{O,\mathcal{L}}$ carries the structure of an irreducible $W$-representation (if nonzero). In this framework, $E_O = V_{O,\CC}$. For $e \in \mathcal{N}$,
    \[
    H^*((G/B)_e) = (\pi_*\CC[\dim \mfg ] )_e  = \bigoplus_{\mathcal{L}, O} \operatorname{IC}(O,\mathcal{L})_e \otimes V_{O, \mathcal{L}}
    \]
    if $E_O$ appears in $H^*((G/B)_e)$, $\operatorname{IC}(O,\mathcal{L})_e \neq 0$ which implies $e \in \overline{O}$.
\end{proof}

Let $H^*(\Lambda \setminus \Fl_\gamma)_{st}$ be the sub $\widetilde{W}$-representation of $H^*(\Lambda \setminus \Fl_\gamma)$ on which $X_*(T)$ acts unipotently.

\begin{proposition}\label{prop:ind_cohomolgy}
    Let $\gamma \in (L^\heartsuit \mfc)_{O_P}^P$, and assume $d_{O_P} = \delta(\gamma) = \dim \Fl_\gamma$. Let $E_{O_P}$ be the $W_P$ representation corresponding to $O_P$. Then the $W$-representations
    \[
    \operatorname{Ind}_{W_P}^W E_{O_P} \text{  and  } H^*(\Lambda \setminus \Fl_\gamma)_{st}
    \]
    share an irreducible representation of $W$.
\end{proposition}
\begin{proof}
    Note that $Z(\gamma)$ acts on $H^*(\Fl_\gamma/\Lambda)$ through the quotient $Z(\gamma)/\Lambda$. We claim that the action of $\CC[X_*(T)]^W \subset \CC[\widetilde{W}]$ factors through the action of $Z(\gamma)/\Lambda$. This follows from a slight modification of the argument in \cite[Section 5.8-5.11]{yun2014spherical}. Let us describe it. In Section \cite[Section 5.8]{yun2014spherical}, the product formula gives a homeomorphism of stacks 
    \[
    \left[( \Lambda \setminus \Fl_{a,0}) \times \Lambda\right]  \overset{P_a^{ker}}{\times} \prod_{x \in \operatorname{Sing}(a)\setminus 0} \Fl^G_{a,x} \to \left[ \widehat{M}^{\operatorname{par}}_{\hat{a},0} /S\right].
    \]
    The local Picard group $P_{a,0}$ acts on  $\left[( \Lambda \setminus \Fl_{a,0}) \times \Lambda\right]$ diagonally: $\Lambda \setminus \Fl_{a,0}$ inherits the action of $P_{a,0}$ and $P_{a,0}$ acts on $\Lambda$ through the quotient.

    Replacing $\Fl_{a,0}$ with $\left[( \Lambda \setminus \Fl_{a,0}) \times \Lambda\right]$ everywhere in \cite[Section 5.9-5.11]{yun2014spherical}, the same arugment gives that the $\CC[X_*(T)]^W$ action on $H^*_c(\Lambda\setminus \Fl_\gamma \times \Lambda) $ factors through the $Z(\gamma)$ action. Note that
    \[
    H^*_c(\Lambda \setminus \Fl_\gamma \times \Lambda) = H^*(\Lambda \setminus \Fl_\gamma) \otimes_{\CC} \CC[\Lambda].
    \]
    By the construction of the affine Springer action, $\CC[X_*(T)]^W$ acts only on the first factor. The claim follows.

    By \Cref{prop:centralizer_invariant}, there is a copy of $E_{O_P}$ in $H^*(\Lambda\setminus\Fl_\gamma)$ fixed by $\CC[X_*(T)]^W$. Thus, as a $\CC[X_*(T)]^W$ module, it is supported above the maximal ideal corresponding to the point $1 \in T//W$. The covering $T \to T//W$ is finite flat and totally ramified above 1. Therefore, if $V$ is a $\CC[X_*(T)]$ module on which $\CC[X_*(T)]^W$ acts trivially, the action of $X_*(T)$ on $V$ must be unipotent. It follows that $H^*(\Lambda\setminus\Fl_\gamma)_{st}$ contains a copy of $E_{O_P}$.

    Let $V'$ be the semisimplification of $H^*(\Lambda \setminus \Fl_\gamma)_{st}$ as a $\widetilde{W}$-module. Let $W_P'$ be the $\widetilde{W}$-module generated by simple reflections, and $W_P$ be its isomorphic image under the projection $\widetilde{W} \to W$. We also view $W_P \subset \widetilde{W}$ by the inclusion of $W \hookrightarrow \widetilde{W}$. The previous paragraph shows that
    \[
    \left[E_{O_P} \colon V'\right]_{W_P'} \neq 0.
    \]
    Since $X_*(T)$ acts trivially on $V'$, it follows that the decomposition of $V'$ into irreducible $W_P'$-representations agrees with the decomposition of $V'$ into irreducible $W_P$-representations. By Frobenius reciprocity,
    \[
    \Hom_{W_P}(\operatorname{Ind}_{W_P}^W E_{O_P}, V') \neq 0.
    \]
    Since $W_P$ is a finite group, its representations are unaffected by semisimplification. We conclude that
    \[
        \Hom_{W_P}\left(\operatorname{Ind}_{W_P}^W E_{O_P}, H^*(\Lambda\setminus\Fl_\gamma)_{st}\right) \neq 0
    \]
    which is what we wanted.
\end{proof}

\begin{remark}
    Let $d = \dim \Fl_\gamma$. Tracing through the proofs of \Cref{prop:centralizer_invariant} and \Cref{prop:ind_cohomolgy}, one obtains the more precise result
    \[
        \Hom_{W_P}\left(\operatorname{Ind}_{W_P}^W E_{O_P}, H^d(\Lambda\setminus\Fl_\gamma)_{st}\right) \neq 0.
    \]
\end{remark}

\section{Global Springer theory}

We recall Yun's global Springer theory before using tools it provides to prove \Cref{thm_KL}.

\subsection{Parabolic Hitchin stack} We recall the definition of the parabolic Hitchin stack and related spaces following \cite[Section 2]{yun2011global}.

Let $X$ be a smooth projective curve of genus $g$. Let $D$ be a divisor on $X$ satisfying $\deg D > 2g$, $D = 2D'$ for some other divisor $D'$.  
\begin{definition}[{\cite[Definition 2.1.1]{yun2011global}}]\label{def:parabolic_hitchin}
    Define $\mathcal{M}^{\operatorname{par}}_{G,X,D}$ to be the functor which sends a scheme $S$ to the groupoid of quadruples $(x,\varepsilon,\varphi, \varepsilon_x^B)$ where
    \begin{itemize}
        \item[--] $x \in \Hom(S,X)$. Denote its graph by $\Gamma(x)$
        \item[--] $\varepsilon$ is a $G$-torsor over $S \times X$
        \item[--] $\varphi \in H^0(S\times X ,\operatorname{Ad}(\varepsilon)\otimes_{\mathcal{O}_X} \mathcal{O}_X(D))$
        \item[--] $\varepsilon_x^B$ is a $B$-reduction of $\varepsilon$ along $\Gamma(x)$. 
    \end{itemize}
\end{definition}

For a $\mathbb{G}_m$-equivariant stack $\mathcal{S}$ over $X$ (with respect to the trivial action on $X$), we define
\[
\mathcal{S}_D \coloneqq \rho_D \overset{\mathbb{G}_m}{\times}_X \mathcal{S} =\left[(\rho_D \times_X \mathcal{S} )/ \mathbb{G}_m\right]
\]
where $\rho_D$ is the complement of the zero section in $\mathcal{O}_X(D)$, and $\mathbb{G}_m$ acts diagonally.

Let $\mathcal{A}^{\operatorname{Hit}} = H^0(X, \mfc_D)$. 
\begin{definition}
    The enhanced Hitchin base $\widetilde{\mathcal{A}}^{\operatorname{Hit}}$ is the pullback of the twisted quotient morphism $\mathfrak{t}_D \to \mathfrak{c}_D$ along the evaluation morphism $X \times \mathcal{A}^{\operatorname{Hit}} \to \mfc_D$:
    \begin{center}
    \begin{tikzcd}
        \widetilde{\mathcal{A}}^{\operatorname{Hit}}\arrow[d,"q"]\arrow[r,"\widetilde{\operatorname{ev}}"] & \mathfrak{t}_D\arrow[d]\\
        \mathcal{A}^{\operatorname{Hit}}\times X\arrow[r, "\operatorname{ev}"] & \mfc_D
    \end{tikzcd}
    \end{center}
\end{definition}

Recall the Hitchin stack $\mathcal{M}^{\operatorname{Hit}}_{G,X,D}$ which sends a scheme $S$ to the groupoid of pairs $(\varphi, \varepsilon)$ where the symbols are as in \Cref{def:parabolic_hitchin}. Evaluating the Higgs field $\varphi$ at a point, we have evaluation morphisms
\[
\operatorname{ev^{par}}\colon \mathcal{M}^{\operatorname{par}} \to [\mathfrak{b}/B]_D \text{  and  } \operatorname{ev^{Hit}}\colon \mathcal{M}^{\operatorname{Hit}} \times X \to [\mfg/G]_D
\]
There is a commutative diagram
\begin{center}
\begin{tikzcd}
     \mathcal{M}^{\operatorname{par}}\arrow[r,"\operatorname{ev^{par}}"] \arrow[d]& \left[\mathfrak{b}/B\right]_D \arrow[d,"\pi_G"]\arrow[r]& \mathfrak{t}_D\arrow[d]\\
     \mathcal{M}^{\operatorname{Hit}}\times X\arrow[r,"\operatorname{ev^{Hit}}"]  & \left[\mathfrak{g}/G\right]_D\arrow[r] & \mathfrak{c}_D

\end{tikzcd}
\end{center}
where $\pi_G$ is the Grothendieck simultaneous resolution for $G$, the left square is Cartesian, and the left arrow forgets the $B$-reduction. Along with the Hitchin map $\mathcal{M}^{\operatorname{Hit}} \to \operatorname{\mathcal{A}^{Hit}}$, this diagram defines
\[
\widetilde{f} \colon\operatorname{\mathcal{M}^{par}} \to \widetilde{\mathcal{A}}^{\operatorname{Hit}}
\]
which we call the enhanced parabolic Hitchin fibration. The composition
\[
f\colon \operatorname{\mathcal M^{par}} \to  \widetilde{\mathcal{A}}^{\operatorname{Hit}}\xrightarrow{q} \mathcal{A}^{\operatorname{Hit}} \times X 
\]
is called the parabolic Hitchin fibration.

From now on, we suppress references to the fixed curve $X$ and divisor $D$.

\subsection{Global Picard} We follow \cite[Section 4.3]{ngo2010lemme}, \cite[Section 2.3]{yun2011global}.

Recall the regular centralizer $J$ from \Cref{sect:local_picard}. For any scheme $S$ and $a \in \mathcal{A}^{\operatorname{Hit}}(S) = \operatorname{Maps}(S \times X ,\mfc_D)$, consider the stack $\mathcal{P}_a$ which classifies $J_a \coloneqq a^* J_D$ torsors on $S \times X$. By \cite[4.3.1]{ngo2010lemme}, as $a$ varies, this defines a Picard stack $\mathcal{P} \to \mathcal{A}^{\operatorname{Hit}}$. 

By \cite[Lemma 2.3.3]{yun2011global}, $\mathcal{P}$ acts on $\mathcal{M}^{\operatorname{par}}$ and preserves the parabolic Hitchin fibration. This action is compatible with the action of $\mathcal{P}$ on $\mathcal{M}^{\operatorname{Hit}}$.

Let $\mathcal{A}^\heartsuit = \operatorname{ev}^{-1}(\mfc^{rs})$. This is an open dense subscheme of $\mathcal{A}^{\operatorname{Hit}}$. Let $\mathcal{A}$ be the open subset of $\mathcal{A}^\heartsuit$ consisting of points $a$ such that $\pi_0(\mathcal{P}_a)$ is finite. Let $\widetilde{\mathcal{A}} = q^{-1}(\mathcal{A}\times X)$.

By \cite[Proposition 2.5.1]{yun2011global}, the restriction $\mathcal{M}^{\operatorname{par}}\mid_{\mathcal{A}}$ is a smooth Deligne-Mumford stack. By \cite[Corollary 2.5.2]{yun2011global}, the parabolic Hitchin fibrations $f$ and $\widetilde{f}$ are flat and proper.

\subsection{Global Springer representations}\label{sect:global_springer_action} \begin{definition}[{\cite[Definition 3.3.2]{yun2011global}}]
    We call $f_*^{\operatorname{par}}\overline{\mathbb{Q}}_{\ell} \in D^b(\mathcal{A} \times X)$ the parabolic Hitchin complex. 

    Similarly, we call $\widetilde{f}_*^{\operatorname{par}}\overline{\mathbb{Q}}_{\ell} \in D^b(\widetilde{\mathcal{A}})$ the enhanced parabolic Hitchin complex.
\end{definition}
In \cite[Theorem 3.3.3, Proposition 3.3.5]{yun2011global}, Yun constructs left $\widetilde{W}$ actions on the parabolic and enhanced parabolic Hitchin complexe using cohomological correspondences. The action on $\widetilde{f}_*^{\operatorname{par}}\overline{\mathbb{Q}}_{\ell}$ is compatible with the $W$-action on $\widetilde{\mathcal{A}}$ over $\mathcal{A}$ via the natural projection $\widetilde{W} \to W$.

In \cite[Theorem 5.1.2, Proposition 5.2.1]{yun2011global}, Yun provides an alternative construction of the aforementioned $\widetilde{W}$ action. The construction is reminiscent of the local Springer action (see \Cref{sect:springer_action}). Let us recap the construction.

For each parahoric subgroup $P \subset LG$, Yun constructs an algebraic stack $\mathcal{M}_P$. It is a parahoric analog of $\mathcal{M}^{\operatorname{par}}$, and $\mathcal{M}^{\operatorname{par}} = \mathcal{M}_I$. They fit into a cartesian diagrams
\begin{equation}\label{diag:global_springer_action}
\begin{tikzcd}
    \mathcal{M}^{\operatorname{par}} \arrow[r]\arrow[d, "\operatorname{For}_I^P"]& \left[\underline{\mathfrak{b}}_P \arrow[r]\arrow[d,"\pi_I^{P_i}"]/ \underline{B}_P\right]_D  & \left[\widetilde{\mfg}_P / G_P^\natural \right]\arrow[d, "\pi_P"]\\
    \mathcal{M}_P\arrow[r] & \left[\underline{\mathfrak{g}}_P / \underline{G}_P\right]_D \arrow[r] & \left[{\mfg}_P / G_P^\natural \right]
\end{tikzcd}
\end{equation}
Here, $\pi_P$ is the Grothendieck simultaneous resolution of $G_P$, and the underlines denote inner forms of $G_P$, etc. over $X$. By classical Springer theory, there is an action of $W_P$ on $\pi_{P,*}\overline{\mathbb{Q}}_\ell$. By proper base change, we get an action of $W_P$ on $\pi_{I,*}^{P_i} \overline{\mathbb{Q}}_\ell$. By another application of base change, we get an action of $W_P$ on
\[
f_*^{\operatorname{par}}\overline{\mathbb{Q}}_\ell = f_*^P \operatorname{For}_{I,*}^P\overline{\mathbb{Q}}_\ell.
\]
Parallel to the local affine Springer action, Yun shows that these actions are compatible, and extend to a $\widetilde{W}$-action on the parabolic Hitchin complex.

For more details, see \cite[Construction 5.1.1]{yun2011global}.

\subsection{Duality} Consider the (enhanced) parabolic Hitchin fibrations of $G$ and $G^\vee$:
\begin{center}
\begin{tikzcd}
    \mathcal{M}_{a,x}^{\operatorname{par}} \arrow[rd,"\widetilde{f}"]\arrow[rdd,"f",swap]& & \mathcal{M}_{a,x}^{\operatorname{par},\vee}\arrow[ld,"\widetilde{f^\vee}",swap]\arrow[ldd,"f^\vee"]\\
    & \widetilde{\mathcal{A}} \arrow[d,"q"]& \\
    & \mathcal{A} \times X. &
\end{tikzcd}
\end{center}

In this section, we use results from \cite{yun2012langlands} to prove:
\begin{proposition}\label{prop:global_isom}
Let $(a,x) \in \mathcal{A} \times X$. There is an isomorphism of $W$-modules
\[
H^*(\mathcal{M}_{a,x}^{\operatorname{par}})_{st} \cong  H^*(\mathcal{M}_{a,x}^{\operatorname{par},\vee})_{st}
\]
\end{proposition}
\begin{lemma}\label{lem:pushforward_springer_action}
    Let $w \in W$. Let 
    \[
    [\mathcal{H}_{w}]_\# \colon f_*^{\operatorname{par}} \CC \to f_*^{\operatorname{par}}\CC
    \]
    be the global Springer action over $\mathcal{A} \times X$ (\cite[Theorem 3.3.3]{yun2011global}), and 
    \[
    [\mathcal{H}_{w}^\natural]_\# \colon \widetilde{f}_*^{\operatorname{par}} \CC \to \widetilde{f}_*^{\operatorname{par}}\CC
    \]
    be the global Springer action over $\widetilde{A}$ (\cite[Proposition 3.3.5]{yun2011global}). Then
    \[
    q_*[\mathcal{H}_{w}^\natural]_\# \colon f_*^{\operatorname{par}} \CC \to f_*^{\operatorname{par}}\CC
    \]
    is the global Springer action $[\mathcal{H}_{w}]_\# $.
\end{lemma}
\begin{proof}
    We are in the following general situation. We have a commutative diagram
    \begin{center}
    \begin{tikzcd}
        & C\arrow[dl,"\overleftarrow{c}",swap] \arrow[dr,"\overrightarrow{c}"] &\\
        X\arrow[r,"\widetilde{f}"]\arrow[rd,"f",swap] & \widetilde{S} \arrow[d,"q"]&\arrow[l,"\widetilde{g}",swap] Y \arrow[ld,"g"]\\
        & S. &
    \end{tikzcd}
    \end{center}
    where the upper triangle and outer square are correspondence diagrams (see \cite[Appendix A.1]{yun2011global}). Let $\mathcal{F} \in D^b(X)$ and $\mathcal{F} \in D^b(Y)$. The composition
    \[
    \operatorname{Corr}(C;\mathcal{F},\mathcal{G}) \coloneqq \Hom_C(\overrightarrow{c}^*\mathcal{G} ,\overleftarrow{c}^! \mathcal{F}) \to \operatorname{Corr}(X \times_S Y; \mathcal{F}, \mathcal{G}) \to \Hom_S(g_!\mathcal{G}, f_*\mathcal{F})
    \]
    produces a map of complexes from a correspondence diagram. The maps above are compositions of adjunction and base change morphisms. From this, it is easy to check that 
    \begin{center}
    \begin{tikzcd}
        & \Hom_{\widetilde S}(g_!\mathcal{G}, f_*\mathcal{F})\arrow[dd,"q_*"]\\
        \operatorname{Corr}(C;\mathcal{F},\mathcal{G})\arrow[ur]\arrow[dr]\\ 
        &  \Hom_{ S}(g_!\mathcal{G}, f_*\mathcal{F})
    \end{tikzcd}
    \end{center}
    commutes.

    In our situation, $C$ is the Hecke correspondence constructed in \cite[Theorem 3.3.3]{yun2011global}, $\mathcal{G} =\mathcal{F} = \CC$, $X = Y = \mathcal{M}^{\operatorname{par}}$, $S = \mathcal{A}\times X$ and $\widetilde S = \widetilde{\mathcal{A}}$
    
\end{proof}

\begin{proof}[Proof of \Cref{prop:global_isom}]
    Let $d = \dim \mathcal{M}^{\operatorname{par}} = \dim \mathcal{M}^{\operatorname{par},\vee}$. Let 
    \[
    K = (\widetilde{f}_*\CC )_{st}[d](d/2) \text{  and  } L = (\widetilde{f^\vee}_*\CC)_{st}[d](d/2)
    \]
    and
    \[
    K^i = {}^p{H^i(K)} \text{ and } L^i = \prescript{p}{}{H^i(L)}.
    \]
    By \cite[Theorem A]{yun2012langlands}, there are natural isomorphisms of perverse sheaves
    \begin{equation}\label{eqn:perverse_isom}
    K^{-i} \cong \mathbb{D}K^i \cong L^i(i).
    \end{equation}
    Note that the global Springer action of $\widetilde{W}$ preserves the perverse filtration. We first show that Yun's isomorphism is $W$-equivariant.

    By \cite[Theorem B]{yun2012langlands}, the perverse sheaves in question are middle extensions from the nice locus $\widetilde{\mathcal{A}}^\diamondsuit \subset \widetilde{\mathcal{A}}$. Therefore it suffices to prove $W$-equivariance over the nice locus. According to \cite[Equations 4.15, 4.16]{yun2012langlands}, we have isomorphisms
    \begin{align*}
    \mathbb{D} K^i_{\diamondsuit} &\cong \tilde{p}^*\bigwedge^{i+N} (V_\ell (\mathcal{P}^0/\mathcal{A}^\diamondsuit))(d/2 - N) \\
    L^i_{\diamondsuit} &\cong  \tilde{p}^*\bigwedge^{i+N} (V_\ell (\mathcal{P}^{\vee,0}/\mathcal{A}^\diamondsuit)^*)(d/2)
    \end{align*}
    where $\tilde{p} \colon \widetilde{\mathcal{A}}^{\diamondsuit} \to \mathcal{A}^\diamondsuit$ is the natural projection, and $V_\ell (\mathcal{P}^0/\mathcal{A}^\diamondsuit)$, $V_\ell (\mathcal{P}^{\vee,0}/\mathcal{A}^\diamondsuit)$ are isomorphic $W$-invariant local systems over $\mathcal{A}^\diamondsuit$ (\cite[Lemma 4.2.1]{yun2012langlands}). Therefore the $W$-module structures on the perverse sheaves above are determined by the $W$-torsor structure of $\widetilde{\mathcal{A}}\to \mathcal{A} \times X $. It follows that the identification $\mathbb{D} K_{\diamondsuit}^i \cong L_{\diamondsuit}^i$ is $W$-equivariant.

    Next, let
    \[
    K' = ({f}_*\CC )_{st}[d](d/2) \text{  and  } L' = ({f^\vee}_*\CC)_{st}[d](d/2),
    \]
    and similarly define $K'^{i}$, $L'^{i}$. 

    The map $q \colon \widetilde{A} \to A \times X$ is finite and flat because it is the base change of $\mathfrak{t} \to \mathfrak{c}$. In particular, $q_*$ is exact for the perverse $t$-structure. Therefore applying $q_*$ to \eqref{eqn:perverse_isom}, we get
    \begin{equation}\label{eqn:pervese_isom_A}
    K'^{-i} \cong L'^{i}(i)
    \end{equation}
    which is equivariant with respect to the $q_*$-pushforward of the $W$-action. By \Cref{lem:pushforward_springer_action}, this is exactly the $W$-part of the global Springer action.

    Let us use this to conclude the proof of the proposition. By the decomposition theorem, $K'$ and $L'$ are semisimple complexes. The perverse filtration defines a $W$-stable filtration
    \[
    P_{\leq n} H^*(\mathcal{M}^{\operatorname{par}}_{a,x})_{st} = ({}^p{\tau_{\leq n + \dim {\mathcal{A} \times X}}}K'[-d])_{(a,x)}
    \]
    whose associated graded quotients are $K'^{n}_{a,x}$ (perhaps with shifts and twists). Since $W$-representations are semisimple, the isomorphism \eqref{eqn:pervese_isom_A} gives the result.
\end{proof}

\subsection{Global to local}

In this section, we use \Cref{prop:global_isom} to finish the proof of \Cref{thm_KL}.

\subsubsection{Global data} Let $X = \mathbb{P}^1$, $\gamma \in (L^\heartsuit \mfc)_{O_P}^P$ be such that $d_{O_P} = \delta(\gamma) = \dim \Fl_\gamma$. By \cite[Proposition 8.2]{kazhdan1988fixed} and \Cref{lem:c_constructible}, the subset of such $\gamma$ is dense and constructible in $(L^\heartsuit \mfc)_{O_P}^P$. By \cite[Proposition 2.9]{yun2014spherical}, there exists $N^G_\gamma>0$ such that for any $\gamma' \in L\mfg$ satisfying $\chi(\gamma') \equiv \gamma \pmod {t^{N^G_\gamma}}$, there exist isomorphisms
\[
\iota\colon \Lambda\setminus \Fl_\gamma \to \Lambda\setminus \Fl_{\gamma'} \text{  and  } \iota_{\mathcal{P}} \colon \mathcal{P}_\gamma \to \mathcal{P}_\gamma 
\]
which induces a $\widetilde{W}$-equivariant isomorphism on cohomology via pullback.  Note that in \cite[Proposition 2.9]{yun2014spherical}, the result is proved for $\Fl_\gamma$. The statement we need follows from the fact that the isomorphism $\iota$ intertwines the local Picard actions via $\iota_{\mathcal{P}}$.

Let $N = \max(N^G_\gamma, N^{G^\vee}_\gamma)$. 

\begin{lemma}\label{lem:section}
    Let $N>0$. There exists a divisor $D$, disjoint from $0 \in \mathbb{P}^1$, such that there exists $a \in \mathcal{A} = \mathcal{A}_D \subset \Gamma(X, \mathfrak{c}_D)$ satisfying
    \[
    a_0  \equiv \gamma \pmod{t^N}
    \]
    where $a_0$ is the image of $a$ in $\mathcal{O}_{X,0}$.
\end{lemma}
\begin{proof}
    By \cite[Proposition 6.5.1]{ngo2010lemme}, the codimension of $\mathcal{A}^\heartsuit \setminus \mathcal{A}$ in $\mathcal{A}^\heartsuit$ is at least $\deg(D)$. Consider the evaluation map at $0$
    \[
    \Gamma(X, \mathfrak{c}_D) \xrightarrow{\operatorname{ev}_{0,N}} \mathfrak{c}(\mathcal{O}_{X,0}/t^N).
    \]
    It is a surjective linear map. Let $V = \operatorname{ev}_{0,N}^{-1}(\gamma)$. Then $\operatorname{codim}_{\Gamma(X,\mfc_D)} V = rN$ (recall that $r$ is the rank of $G$). Since $\mathcal{A}^\heartsuit$ is open dense in $\Gamma(X, \mathfrak{c}_D)$, for any $D$ of degree at least $rN+1$, $V$ must intersect $\mathcal{A}$. 
\end{proof}
Let $D$ be the divisor provided by \Cref{lem:section} with $N$ as in the beginning of this section. We use this to define the parabolic Hitchin stacks for $G$ and $G^\vee$. Let $a$ be as in \Cref{lem:section}. 

\subsubsection{Product formula} Let $\operatorname{Sing}(a)$ be the set of points $x'$ such that $a(x') \notin \mfc^{rs}$, along with $0$. For $x' \in \operatorname{Sing}(a)$, fix a trivialization of $\mathcal{O}_{X,x'}$. The global Kostant section $\kappa \colon \mathcal{A} \to \mathcal{M}^{\operatorname{Hit}}$ gives a Hitchin pair, hence an element $\gamma(a,x) \coloneqq a_{x'} \in \mathfrak{c}(\mathcal{O}_{X,x'})$.

By the product formula(\cite[Proposition 2.4.1]{yun2011global}, \cite[Proposition 4.13.1]{ngo2010lemme}),
we have a homeomorphism of stacks
\[
\mathcal{P}_a \overset{\mathcal{P}_{\gamma(a,0)}^{\operatorname{red}}(J_{\gamma(a,0)})\times P'}{\times}\left(\Fl_{\gamma(a,0)} \times M'\right)\xrightarrow{\sim}\mathcal{M}^{\operatorname{par}}_{a,0}
\]
where
\[
P' = \prod_{x' \in \operatorname{Sing}(a) - \{0\}} \mathcal{P}_{\gamma(a,x')}^{\operatorname{red}}(J_{\gamma(a,x')}) \text{  and  } M' = \prod_{x'\in \operatorname{Sing(a)}-\{0\}} \Fl^G_{\gamma(a,x')}.
\]
For each local regular centralizer $J_\gamma$, let $J^\flat_\gamma$ its finite type N\'{e}ron model. Define $P_\gamma$ by replacing $J_\gamma$ in the definition of $\mathcal{P}_\gamma$ by $J^\flat_\gamma$. By \cite[Lemma 3.8.1]{ngo2010lemme}, the reduced structure of $P_\gamma$ is a free abelian group of finite type. By the same lemma, the local Picard stack fits in a short exact sequence
\[
1 \to R_\gamma \to \mathcal{P}_\gamma^{\operatorname{red}} \to P_\gamma^{\operatorname{red}}\to 1.
\]
Similarly, let $J_a^\flat$ be the N\'{e}ron model of $J_a$ and the corresponding Picard stack $P_a$. By \cite[Proposition 4.8.2]{ngo2010lemme}, there is an essentially surjective morphism $r\colon \mathcal{P}_a \to P_a$ with
\[
\ker(r) = \prod_{x'\in \operatorname{Sing}(a)} R_{\gamma(a,x')}.
\]
Therefore, there is a locally trivial fibration
\[
(\Lambda \setminus \Fl_{\gamma(a,0)}) \times M \to \mathcal{M}^{\operatorname{par}}_{a,0} \to A 
\]
where
\[
M = \prod_{x' \in \operatorname{Sing}(a) - \{0\}} \Lambda_{\gamma(a,x')}\setminus \Fl^G_{\gamma(a,x')}
\]
Restricting the diagram \eqref{diag:global_springer_action} to the point $(a,0)$, for any parahoric subgroup $P$, we have a commutative diagram
\begin{center}
\begin{tikzcd}
    (\Lambda \setminus \Fl_{\gamma(a,0)}) \arrow[d,"\pi_P^I"]\times M\arrow[r,hookrightarrow] & \mathcal{M}^{\operatorname{par}}_{a,0}\arrow[d,"\operatorname{For}_I^P"] \arrow[r]& \left[\widetilde{\mfg}_P/G_P\right]\arrow[d, "\pi_{\mfg_P}"]\\
    (\Lambda \setminus \Fl_{\gamma(a,0)}^P) \times M\arrow[r,hookrightarrow] & \mathcal{M}_{P,a,0} \arrow[r]& \left[{\mfg_P}/G_P\right]
\end{tikzcd}
\end{center}
where the composition of the horizontal arrows are the evaluation maps for the (parahoric) affine Springer fibers. One has similar diagrams for the action of $\omega \in \Omega$.

Therefore, we have
\begin{lemma}\label{lem:spectral_seq}
    There is a spectral sequence
    \[
    E_2^{p,q} = H^p\left(A , H^q(\Lambda \setminus \Fl_{\gamma(a,0)} \times M)\right) \Rightarrow H^{p+q}(\mathcal{M}_{a,0}^{\operatorname{par}})
    \]
    which intertwines the $\widetilde{W}$ action with respect to the global Springer action on $H^{p+q}(\mathcal{M}_{a,0}^{\operatorname{par}})$, and the affine Springer action on $H^q(\Lambda \setminus \Fl_{\gamma(a,0)})$.
\end{lemma}
\begin{proof}[Proof of \Cref{thm_KL}]
    By \Cref{prop:ind_cohomolgy}, there is some $W$-representation, say $E$, appearing in both $\operatorname{Ind}_{W_P}^W E_{P}$ and $H^{\operatorname{top}}(\Lambda\setminus \Fl_\gamma)_{st}$ (recall $E_{P}$ is the special representation corresponding to $O_P$). By choice of $a$, and \Cref{lem:spectral_seq}, $E$ appears in
    \[
    H^{\operatorname{top}}(\mathcal{M}^{\operatorname{par}}_{a,0}) = H^{\operatorname{top}}(A) \otimes H^{\operatorname{top}}(\Lambda \setminus \Fl_{\gamma(a,0)} \times M).
    \]
    By $\widetilde{W}$-equivariance, $E$ in fact lies in $H^*(\mathcal{M}^{\operatorname{par}})_{st}$. By \Cref{prop:global_isom}, $E$ appears in $H^*(\mathcal{M}^{\operatorname{par},\vee})_{st}$, in particular, in $H^*(\mathcal{M}^{\operatorname{par},\vee})_{st}$.

    By \Cref{lem:spectral_seq}, $E$ appears in
    \[
    H^*(\Lambda \setminus \Fl_{\gamma(a,0)}^{G^\vee}) \cong H^*(\Lambda \setminus\Fl_\gamma^{G^\vee}).
    \]
    Let $E$ correspond to the pair $(O, \mathcal{F})$ where $O$ is a nilpotent orbit of $G^\vee$ and $\mathcal{F}$ is a local system on $O$. Let $e \in O$. Then
    \begin{alignat*}{2}
        \dim (G^\vee /B^\vee)_e &\geq a_{E_P}  &&\text{    (by \cite[Corollary 10.5(i)]{lusztig1992unipotent})}\\
        &= b_{E_P}  && \text{ (since }O_P\text{ is special)}\\
        &= d_{O_P} \\
        &= \dim \Fl_\gamma &&\text{ (by choice of }\gamma\text{)}\\
        &= \dim \Fl_\gamma^{G^\vee} &&\text{ (by \Cref{lem:fl_dual_same_dim})}.
    \end{alignat*}
    By the proof of \Cref{prop:centralizer_invariant}, the $W$-representations appearing in $H^*(\Lambda \setminus \Fl_{\gamma}^{G^\vee})$ are a subset of those corresponding to representations appearing in the Springer fibers of various reduction types of $\Fl_\gamma^{G^\vee}$. By \Cref{lem:springer_smaller_orbit}, if $E$ appears in the cohomology of some reduction type $O'$, then 
    \[
    \dim(G^\vee /B^\vee)_{e'} \geq \dim (G^\vee /B^\vee)_e
    \]
    where $e' \in O'$. On the other hand, we must have $\dim \Fl_\gamma^{G^\vee} \geq \dim(G^\vee /B^\vee)_{e'}$. Therefore we conclude that
    \[
    \dim (G^\vee /B^\vee)_e =a_{E_P}.
    \]
    By \cite[Corollary 10.5(ii)]{lusztig1992unipotent}, $E = j_{W_P}^W E_P$. This is exactly what we wanted to show.
\end{proof}

\section{A generalization and consequences}

We note that the methods of this paper in fact prove a strengthening of \Cref{thm_KL}. Let us briefly introduce some notation, following \cite[Section 1.1]{lusztig2015conjugacy}.

Let $s\in T$ be semisimple. There exists $x \in \mathfrak{t}_\mathbb{R}$ such that the set of roots taking integral values on $x$ gives the root system for the pseudo-Levi subgroup $Z_G(s)^\circ \coloneqq G_x$. Denote the Weyl group of this reductive group by $W_x$.

According to \cite[Theorem 1.5]{LUSZTIG20093418}, unipotent orbits in $Z_G(s)^\circ$ are obtained by truncated induction of special representations in its endoscopic groups. That is, for any unipotent class $u \subset Z_G(s)^\circ$, there exists $\zeta \in \mathfrak{t}_\mathbb{R}^*$ such that coroots of $Z_G(s)^\circ$ taking integral values on $\zeta$ correspond to roots of the root system of a reductive group $G_{x,\zeta}$ with Weyl group $W_{x,\zeta}$ and a special representation $E_{x,\zeta}$ of $W_{x,\zeta}$ such that $j_{W_{x,\zeta}}^{W_x} E_{x,\zeta}$ corresponds to $u$ under the Springer correspondence.

Note that $W_\zeta$ can be viewed as the Weyl group of a pseudo-Levi subgroup $G^\vee_\zeta \subset G^\vee$ (given by similar integral root conditions). Then $W_{x,\zeta}$ can be viewed as the Weyl group of an endoscopic group for $W_\zeta$. Therefore $j_{W_{x,\zeta}}^{W_\zeta} E_{x,\zeta}$ corresponds to a unipotent element in $G_\zeta^\vee$.

Now, we can take $x, \zeta$ to lie in the compact spherical alcoves for $G$ and $G^\vee$. We get corresponding parahoric subgroups $P_x \subset LG$, $Q_\zeta \subset LG^\vee$ such that $P_x / P_x^+ \cong G_x$ and $Q_\zeta / Q_\zeta^+ \cong G^\vee_\zeta$.  

\begin{theorem}\label{thm:generalized_KL}
    With $x, \zeta, P_x, Q_\zeta$ as above, and $E_{x,\zeta}$ any special representation of $W_{x,\zeta}$,
    \[
    \operatorname{KL}_{G^\vee}^{Q_\zeta} \left(\operatorname{Spr}_{G^\vee_\zeta} j_{W_{x,\zeta}}^{W_\zeta} E_{x,\zeta}\right) = \operatorname{KL}_G^{P_x} \left(\operatorname{Spr}_{G_x} j_{W_{x,\zeta}}^{W_x} E_{x,\zeta}\right).
    \]
\end{theorem}

\begin{proof}
    Let $\gamma \in P_x$ lift $\operatorname{Spr}_{G_x} j_{W_{x,\zeta}}^{W_x} E_{x,\zeta}$. By \Cref{prop:ind_cohomolgy}, $\Hom_{W_x} ( j_{W_{x,\zeta}}^{W_x} E_{x,\zeta}, H^*(\Lambda \setminus \Fl_\gamma)_{st})\neq 0$. By Frobenius reciprocity,
    \begin{align*}
    \Hom_{W} ( \operatorname{Ind}_{W_x}^Wj_{W_{x,\zeta}}^{W_x} E_{x,\zeta}, H^*(\Lambda \setminus \Fl_\gamma)_{st}) \neq 0.
    \end{align*}
    Therefore,
    \[
    \Hom_{W} (\operatorname{Ind}_{W_x}^Wj_{W_{x,\zeta}}^{W_x} E_{x,\zeta}, H^*(\Lambda \setminus \Fl_\gamma^\vee)) \neq 0
    \]
    where we argue as in the proof of \Cref{thm_KL}. Again by Frobenius reciprocity,
    \[
    \Hom_{W} (\operatorname{Ind}_{W_{x,\zeta}}^W E_{x,\zeta}, H^*(\Lambda \setminus \Fl_\gamma^\vee)) = \Hom_{W_\zeta} (\operatorname{Ind}_{W_{x,\zeta}}^{W_\zeta} , \operatorname{Res}_{W_\zeta}^WH^*(\Lambda \setminus \Fl_\gamma^\vee)) \neq 0.
    \]
    Let $E$ appear in $\operatorname{Ind}_{W_{x,\zeta}}^{W_{\zeta}}E_{x,\zeta}$. Then according to \cite[Corollary 10.5]{lusztig1992unipotent}, either $E = j_{W_{x,\zeta}}^{W_\zeta} E_{x,\zeta}$, or $\operatorname{dim}(G^\vee_\zeta / B^\vee_\zeta)_e > a_{E_{x,\zeta}}$, where $e$ is any nilpotent in the orbit corresponding to $E$ (perhaps not attached to the constant local system) under Springer's correspondence. The second case is impossible, as can be seen by a dimension argument as in the proof of \Cref{thm_KL}.
\end{proof}
In particular,
\begin{corollary}
    Let $S_G \subset \underline{W}$ be the union of all $\operatorname{KL}_G^P(u)$ for various $P$, $u \in P/P^+$. Then $S_G = S_{G^\vee}$.
\end{corollary}
\begin{remark}
    \Cref{thm_KL} is the special case of \Cref{thm:generalized_KL} where $\zeta =1$.
\end{remark}

\subsection{Strata} In \cite{lusztig2015conjugacy}, Lusztig defines a partition of $G$ into subsets called strata. Each stratum is a union of conjugacy classes. This partition has desirable properties. For example, each stratum is locally closed (see \cite{carnovale2020lusztig}). Let us briefly recall this definition. 

Let $g = su \in G$ be the Jordan decomposition of an element. Let $W_s$ be the Weyl group of $Z_G(s)^\circ$. Under Springer's correspondence, $u$ corresponds to a representation $E_u$ of $W_s$. We associate the representation $j_{W_s}^W E_u$ of $W$ to $g$.

\begin{definition}[\cite{lusztig2015conjugacy}]
    This assignment gives a well defined map $G \to \operatorname{Rep}(W)$. The fibers of this map are called the strata of $G$.
\end{definition}

In \cite[Section 6.4]{lusztig2015conjugacy}, Lusztig conjectures that strata can alternatively be defined as the fibers of parahoric Kazhdan-Lusztig maps. We provide some evidence to support this conjecture.

Let $g \in su \in G$ be such that $u$ is special in $Z_G(s)^\circ$. We denote by $P_x$ a parahoric subgroup such that $P_x/P_x^+$ is identified with $Z_G(s)^\circ \subset G$. Let $W_x$ denote the Weyl group of the Levi quotient of $P_s$. We require that the inclusion $W_x \hookrightarrow W$ agrees with $W_s \hookrightarrow W$ up to an inner automorphism of $W$. In particular, under this hypothesis, $j_{W_x}^W E_u = j_{W_s}^W E_u$.

We claim that $\operatorname{KL}_G^{P_x}(E_u)$ is independent of the choice of $P_x$. Indeed, let $P_{x'}$ be another parahoric subgroup satisfying the above requirements. According to \Cref{thm_KL},
\[
\operatorname{KL}_G^{P_x}(u) = \operatorname{KL}_{G^\vee}(\operatorname{Spr}j_{W_x}^W E_u) = \operatorname{KL}_{G^\vee}(\operatorname{Spr}j_{W_{x'}}^W E_u) = \operatorname{KL}_G^{P_{x'}}(u).
\]
This gives a well defined map
\[
G \supset S \coloneqq \{g = su \mid u \text{ is special in } Z_G(s)^\circ\} \xrightarrow{\varphi} \underline{W}.
\]
Since $\operatorname{KL}_{G^\vee}$ is injective, it follows that 
\begin{corollary}
    The fibers of $\varphi$ are exactly the strata of $G$ contained in $S$.
\end{corollary}

\bibliography{citations}{}
\bibliographystyle{plain}

\end{document}